\documentclass[11pt]{amsart}
\usepackage{graphicx} 
\usepackage{csquotes}
\usepackage[scr=boondoxupr]{mathalpha}
\usepackage{orcidlink}
\usepackage[toc]{appendix}
\usepackage[backend=biber,maxbibnames=99,maxcitenames=2,style=alphabetic]{biblatex}

\renewbibmacro{in:}{}
\usepackage{setspace} 
\DeclareFieldFormat[article]{volume}{\mkbibbold{#1}}
\renewbibmacro*{volume+number+eid}{%
  \printfield{volume}%
  \iffieldundef{number}
    {}
    {\addcomma\space \printtext{no.~\printfield{number}}}%
  \setunit{\addcomma\space}%
  \printfield{eid}%
}

\usepackage{geometry}
\usepackage{amsmath,amsthm,amssymb,amsfonts,mathtools}
\usepackage{mathrsfs}
\usepackage{esint}
\numberwithin{equation}{section}

\usepackage{float}
\usepackage{tikz}       
\usepackage{pgfplots}
\usepackage{caption}
\usepackage{subcaption}
\usepackage{tikz-3dplot}
\usetikzlibrary{decorations.markings}

\usepackage{comment}
\usepackage{indentfirst}
\usepackage{url}

\usepackage{hyperref}
\hypersetup{
    colorlinks,
    citecolor=blue,
    filecolor=black,
    linkcolor=red,
    urlcolor=black
}

\newcommand{\lap}{\Delta}
\newcommand{\RR}{\mathbb{R}}

\newcommand{\spt}{\text{spt}}
\newcommand{\dv}[1]{\,\text{d}#1}
\newcommand{\grad}{\nabla}
\newcommand{\sff}{\mathrm{I\!I}}

\newcommand{\essinf}{\text{ess inf}}

\newcommand\vol{\text{Vol}}
\newcommand\lip{\mathrm{Lip}}

\newcommand{\hess}{\mathbf{Hess}}
\newcommand{\tr}{\text{Tr}}

\newcommand{\HH}{\mathcal{H}}
\newcommand{\Tan}{\mathbf{Tan}}
\newcommand{\nor}{\mathbf{Nor}}
\newcommand{\QHE}{\textup{Qheightex}}
\newcommand{\HE}{\textup{Heightex}}

\newcommand{\Diff}[2]{\mathcal{DIFF}^{#1}_{#2}}
\newcommand{\iv}{\mathcal{IV}}
\newcommand{\intvar}[1]{\dv||V||(#1)}
\newcommand{\LL}{\mathcal{L}}

\newcommand{\Lcorner}{{\mathrel{\makebox[7pt][c]{\rule{.4pt}{7.5pt}\rule{5pt}{.4pt}}}}}

\newtheoremstyle{mainthm}{}{}{\itshape}{}{\bfseries}{.}{.5em}{\thmnote{#3}}
\theoremstyle{mainthm}
\newtheorem{maintheorem}{Theorem}

\theoremstyle{plain}
\newtheorem{theorem}{Theorem}[section]
\newtheorem{lemma}[theorem]{Lemma}
\newtheorem{corollary}[theorem]{Corollary}

\newtheorem{proposition}[theorem]{Proposition}

\theoremstyle{remark}
\newtheorem{remark}{Remark}[section]

\theoremstyle{definition}
\newtheorem{definition}{Definition}[section]

\title[Alexandrov's Theorem for Integral Varifolds]{Alexandrov's Theorem for Integral Varifolds and Applications to Geometric Inequalities}
\author{Mitchell Gaudet \orcidlink{0009-0008-7521-8297}}
\address{University of Toronto, Department of Mathematics}
\email{mitchell.gaudet@mail.utoronto.ca}
\date{\today}

\begin{document}

\maketitle

\begin{abstract}
    In this paper, we use optimal mass transportation to derive an isoperimetric inequality and an interior Michael-Simon-Sobolev inequality for weakly twice differentiable integral varifolds of dimension and codimension at least $2$, with locally bounded first variation and $L^2_{loc}$ mean curvature. This generalizes a result of Brendle and Eichmair \cite{Brendle_2023}. In both inequalities the corresponding constants are optimal in codimension 2. As part of this process we also generalize Alexandrov's theorem for convex functions to this setting. This necessitates the use of measure and distribution-valued solutions to certain geometric differential operators, in addition to the theory of $L^p$ Taylor expansions started by \cite{Calder_n_1961}. 
\end{abstract}

\section{Introduction}
The isoperimetric inequality is a classical problem in differential geometry and the calculus of variations which can be roughly stated as follows: Given some notions $A$ (resp. $V$) of area (resp. volume) of some class of objects $B$ (resp. $O$), with some definition of boundary operator $\partial: O\rightarrow B$, are there constants $C,\lambda>0$ such that $(V(S))^\lambda\leq C A(\partial(S))$ for all $S\in O$?

Notice here that using the natural notion of volume of a 1-dimensional manifold, and area of a 0 dimension point, an isoperimetric inequality cannot hold for 1-dimensional manifolds with boundary.

On the other hand, the (interior) Michael-Simon-Sobolev inequality \cite{Michael_1973} can be stated as follows: Given a generalized class $O$ of "generalized manifolds" and for each $S\in O$ define volume measures $\mu_S$, gradients $\delta_S$ acting on smooth functions, and an assignment $O\ni S\mapsto H(S,*)$ of a "generalized curvature" function, does there exist constants $C,\lambda>0$ such that for all smooth, non-negative, and compactly supported functions $f$, such that the inequality $$\left(\int f^{1/\lambda}\dv\mu_S\right)^{\lambda}\leq C\int |\delta_Sf|+f|H(S,*)|\dv\mu_S$$ holds for any $S\in O$?

The boundary version of the Michael-Simon-Sobolev inequality has the same setup, except now we also fix a boundary measure $\nu_S$ for each $S\in O$, and require instead that the inequality $$\left(\int f^{1/\lambda}\dv\mu_S\right)^{\lambda}\leq C\left[\int |\delta_Sf|+f|H(S,*)|\dv\mu_S+\int f\dv\nu_S\right]$$ holds.

If the gradient of any constant function is 0, and we shrink $O$ to also require that the curvature function is identically 0, and finally each $\mu_S,\nu_S$ is compactly supported, then by taking $B$ to be the set of Radon measures (or closed sets by looking at their supports) then a boundary Michael-Simon-Sobolev inequality implies a corresponding isoperimetric inequality.

Moreover, for all of these a natural question to ask is, after fixing $\lambda$, what is the smallest $C>0$ satisfying these inequalities?

It is known that by \cite{43ed2540-f618-3c80-86f7-002da1c14be5} when taking the set $O$ to be integral currents (resp. polyhedral chains, rectifiable currents) of dimension $n$, the set $B$ to be integral currents (resp. polyhedral chains, rectifiable currents) of dimension $n-1$, the operator $\partial$ to be the natural boundary operator, $A,V$ to be the masses, and $\lambda=\frac{n-1}{n}$, the best constant is equal to $\frac{Area(\mathbb S^{n-1})}{\left(\vol(B_1(0))\right)^{(n-1)/n}}$, and in the case the current is rectifiable or integral equality is achieved by some element of $O$.

For subclasses of $n$-varifolds various Michael-Simon-Sobolev inequalities have been found with non-constructive and non-optimal constant $C>0$, with $\lambda=\frac{n-1}{n}$. For this see \cite{allard1972first} for the area integrand, and \cite{dephilippis2026michaelsimoninequalityanisotropicenergies,firester2026anisotropicmichaelsimoninequality,firester2026anisotropicmichaelsimoninequalitysurfaces} for various conditions on anisotropic integrands. In all of these, the boundary measure is the singular part of the (an)isotropic first variation.

Although we have framed this as having the smallest constant on the right-hand side, as is standard in most of the literature, we will follow the convention of \cite{brendle2021isoperimetric,Brendle_2023}, so our constants will be on the left, and we will want to find the \textit{largest} $C$. 

\subsection{Interior Inequalities}

When taking $O$ to be the set of smooth $n$-dimensional submanifolds of $\RR^{n+m}$ with boundary, and $B$ to be the set of smooth $(n-1)$-dimensional submanifolds of $\RR^{n+m}$, with $\partial,V,A$ being the natural boundary, volume, and area, optimal transportation methods have been used to great effect. In \cite[Section~1.6]{guillen2010lecturesoptimaltransportationgeometry} a short proof in the case of $m=0$ is given. For general $m\geq 0$ \cite{CASTILLON201079} uses optimal transportation to a linear subspace on $\RR^{n+m}$ to get a non-sharp interior Michael-Simon-Sobolev, and a sharp weighted one, and then deduce an isoperimetric inequality from that.

More recently, Brendle \cite{brendle2021isoperimetric} gave a proof of a boundary Michael-Simon-Sobolev inequality using an ABP estimate, and later Brendle and Eichmair \cite{Brendle_2023} gave a proof of the isoperimetric inequality in codimension at least 2 using optimal transport. In codimension $2$ the constants obtained are optimal.

The first result of this paper is a generalization of the interior Michael-Simon-Sobolev using the methods of \cite{Brendle_2023} to the class of integral varifolds with locally bounded first variation, locally $L^2$ mean curvature, and such that the support is weakly twice-differentiable (cf. definition \ref{def:twicediff}).

In the next several theorems, due to this paper using the argument in section 3 of \cite{Brendle_2023} we are restricted to $n,m\geq 2$. In other words, the dimension and codimension are at least 2. Although our methods in sections 6 and 7 rely on Brendle and Eichmair's argument, it would be interesting to see if optimal transport can be used for codimension 1.

We set $$C_{n,m}=n\left(\frac{(n+m)\vol(B^{n+m}_1(0))}{m\vol(B^{m}_1(0))}\right)^{1/n}$$ as in \cite[Corollary~2]{Brendle_2023} to get:

\begin{maintheorem}[Theorem 1]
    \label{Thm:Theorem1}
    Let $V\in\iv_n(\RR^{n+m})$, with $n,m\geq 2$, be such that
    \begin{enumerate}
        \item The first variation is a Radon measure.
        \item $\spt(||V||)$ is compact.
        \item If we define $W=\spt(||V||)-\spt(||\delta V||_{sing})$ then $H(V,*)\in L^2_{loc}(||V||\Lcorner W,\RR^{n+m})$.
        \item $V$ is weakly twice-differentiable.
    \end{enumerate}
    Moreover, assume $0\leq F\in C^\infty(\RR^{n+m})$ has support compactly contained in $W$.

    Then \begin{align*}
        C_{n,m}\left(\int F^{n/(n-1)}\intvar{x}\right)^{(n-1)/n}&\leq\int |\grad^VF|(x)+F(x)|H(V,x)|\intvar{x}
        \\
        &=\int |\grad^VF|\dv||V||+\int F\dv||\delta V||.
    \end{align*}
\end{maintheorem}

In the context of varifolds $V$ satisfying the assumptions $1-4$ this is optimal when $m=2$.

The idea to prove this theorem is essentially the same as \cite{Brendle_2023}, which is done in section 6, while the theorem is proved in section 7. However, in adapting it over we must reprove Alexandrov's theorem (cf. \cite[Theorem~14.1]{Villani_2009}) on the distributional Laplacian and twice-differentiability of convex functions to our setting, which is what sections 2-5 consist of.

Section 2 covers the necessary background material on notation, measure theory and optimal transport, while section 3 contains background and useful material on varifolds which are $C^2$ rectifiable in the sense of \cite{anzellotti1994k}, in addition to the basics of twice weakly differentiable varifolds.

Section 4 then shows that if a function possesses a Taylor expansion in the sense of Calder\'on-Zygmund \cite{Calder_n_1961}, then certain measure-valued differential equations can be represented in terms of this expansion. This is then applied to the case of the Laplacian on a varifold in section 5 to deduce Alexandrov's theorem.

The result \cite[Theorem~19.11]{menne2024sharplowerboundmean} (cf. also \cite[Corollary~19.13]{menne2024sharplowerboundmean}) combined with the fact that a $C^2$ submanifold is weakly twice differentiable at all of its points immediately yields that if we can upgrade $L^2_{loc}$ to $L^p_{loc},p>n$ then the weak twice differentiability assumption holds.

Therefore we also have,
\begin{corollary}
    Suppose that $V\in\iv_n(\RR^{n+m})$, with $n,m\geq 2$, is such that
    \begin{enumerate}
        \item The first variation is a Radon measure.
        \item $\spt(||V||)$ is compact.
        \item In the notation of theorem \ref{Thm:Theorem1}, $H(V,*)\in L^p_{loc}(||V||\Lcorner W,\RR^{n+m}),p>n$.
    \end{enumerate}
    Then the conclusions of theorem \ref{Thm:Theorem1} hold.
\end{corollary}

\subsection{Inequalities at the Boundary}

One problem that was mentioned in \cite{43ed2540-f618-3c80-86f7-002da1c14be5} was: what is the right boundary set and area functional in the case of varifolds, for the purposes of optimal constants in the isoperimetric inequality?

It is clear from the prior theorem that the boundary should lie on $\spt(||\delta V||_{sing})$ , but it is not clear what the mass of this set is. Taking inspiration from another isoperimetric inequality due to Federer, \cite[Theorem~3.2.43]{Federer_1996}, we define the boundary measure using a modified Minkowski content.

We define $$||\partial V||^+(F)\coloneqq\limsup_{r\downarrow0}\frac{1}{r}\int_{0<dist(x,\spt(||\delta V||_{sing}))<r} F\intvar{x}$$ and $$||\partial V||^-(F)\coloneqq\liminf_{r\downarrow0}\frac{1}{r}\int_{0<dist(x,\spt(||\delta V||_{sing}))<r} F\intvar{x}.$$

We then get a boundary Michael-Simon-Sobolev inequality using the Hahn-Banach theorem, which takes up the content of section 9:

\begin{maintheorem}[Theorem 2]
    \label{Thm:Theorem2}
    Let $V$ be as in theorem \ref{Thm:Theorem1}. Suppose that $||\partial V||^+(1)<\infty$. Then there exists a Radon measure $\gamma^V$ such that:\begin{enumerate}
        \item $\spt(\gamma^V)\subseteq\spt(||\delta V||_{sing})$.
        \item For any continuous function $F$ both $||\partial V||^-(F),||\partial V||^+(F)$ are finite, and $$||\partial V||^-(F)\leq\gamma^V(F)\leq||\partial V||^+(F).$$
        \item If $F$ is moreover non-negative and smooth, then \begin{align*}
            C_{n,m}\left(\int_{\RR^{n+m}-\spt(||\delta V||_{sing})} F^{n/(n-1)}\intvar{x}\right)^{(n-1)/n}&\leq\int |\grad^VF|(x)+F(x)|H(V,x)|\intvar{x}
            \\
            &+\int F\dv\gamma^V.
        \end{align*}
    \end{enumerate}
\end{maintheorem}

This yields an isoperimetric inequality if we set the area of the boundary to be equal to $\gamma^V(\spt(||\delta V||_{sing}))$, and the volume to be $||V||(\RR^{n+m}-\spt(||\delta V||_{sing}))$, and in codimension 2 this inequality is optimal.

It is unclear here how this mass relates to the mass of the singular part of the first variation in general. However, using \cite{Allard_1975bnd} we may deduce a relationship in a special case. Let $B$ be a compact $C^2$ $(n-1)$-dimensional submanifold of $\RR^{n+m}$ and define $d_B=dist(*,B)$.

To actually do the estimates, we will take inspiration from the case where $V$ is just a smooth manifold with boundary $B$. In this case, the following $G$ is a smooth extension of the Riemannian distance on $V$ to $B$, and $\varepsilon_G=1$.

\begin{definition}
    A modified normal field for the pair $(V,B)$ is a $C^1$ function $G:\RR^{n+m}\rightarrow\RR$ such that:
    \begin{enumerate}
        \item Defining $\essinf$ as the essential infimum with respect to the measure $||V||$ we have$$\varepsilon_G\coloneqq\lim_{r\downarrow 0}[\essinf(1_{\{0<d_B(*)<r\}}\grad^Vd_B\cdot DG)]>0$$
        \item $DG(x)$ has unit length for all $x\in B$.
    \end{enumerate}
\end{definition}

\begin{maintheorem}[Theorem 3]
    \label{Thm:Theorem3}
    If everything is as in theorem \ref{Thm:Theorem2} and moreover $\spt(||\delta V||_{sing})$ is equal to $B$, then $||\partial V||^+=||\partial V||^-$, $||V||(B)=0$, and the measure $\gamma^V$ is unique. In particular,
    $$\int_{\RR^{n+m}-\spt(||\delta V||_{sing})} F^{n/(n-1)}\intvar{x}=\int_{\RR^{n+m}} F^{n/(n-1)}\intvar{x}.$$

    If, moreover, $\varepsilon_G>0$ then $\gamma^V\leq \frac{1}{\varepsilon_G}||\delta V||_{sing}$.
\end{maintheorem}
The proof of the needed estimates is covered in section 8. In the notation of Allard our $\gamma^V$ is equal to $\gamma_1$ in \cite[Section~2]{Allard_1975bnd}.

Theorem \ref{Thm:Theorem3}, and the remarks preceding it, allow us to immediately conclude with one last isoperimetric inequality, which is once again optimal in codimension 2.

\begin{corollary}
    Let $B$ be a compact, orientable, smooth, $(n-1)$-dimensional submanifold of Euclidean space $\RR^{m+n}$, with $n,n\geq 2$, and suppose that $V$ is an integral $n$-varifold which is supported on a smooth manifold with boundary $B$ with multiplicity $1$ near $B$, and is stationary on the complement of $B$. Then $$C_{n,m}||V||^{(n-1)/n}(\RR^{n+m})\leq \HH^{n-1}(B).$$
\end{corollary}

Finally, two appendices cover the necessary modifications of \cite[Theorems~14.1,14.25]{Villani_2009} to the case of $C^2$ submanifolds of Euclidean space, and to properties of graph coordinate systems on such submanifolds. To the author's knowledge these are both well-known to the experts, but we will have need of these specific formulations.

\textbf{Acknowledgments:} The author would like to thank his supervisor Prof. Yevgeny Liokumovich for suggesting this problem and many helpful discussions. The author was supported by the Ontario Graduate Scholarship. 

\section{Preliminaries}
\subsection{Notation}
We include here some potentially non-standard notation for later use.
\begin{itemize}
    \item $\mathcal{RM}(U)(\text{resp. }\mathcal{RM}(U,\RR^m))$: The space of Radon measures on $U$(resp. $\RR^m$-valued measures on $U$).
    \item $L^p(\mu,\hom(\RR^k,\RR^l)),\mu\in\mathcal{RM}(U)$: The space of functions $U\rightarrow \hom(\RR^k,\RR^l)$ with $L^p(\mu)$ coefficients. $L^p_{loc}(\mu,\hom(\RR^k,\RR^l))$ is defined in the obvious way.
    \item $\Diff{\leq m,p}{k,l}(\mu),\mu\in\mathcal{RM}(U)$: The space of $\RR$-linear combinations of differential operators of the form $A_\alpha D^\alpha$ where $\alpha$ is a multi-index with $|\alpha|\leq m$, and $A_\alpha\in L^p(\mu,\hom(\RR^k,\RR^l))$. $\Diff{\leq m,p,loc}{k,l}(\mu)$ is defined in the obvious way.
    \item $\mathcal{D}(U,\RR^m)$: The space of smooth functions $U\rightarrow\RR^m$ with support a compact subset of $U$.
    \item $\iv_k(U)$: The space of integral $k$-varifolds on $U$.
    \item $S:T$, $S,T\in\hom(\RR^k,\RR^k)$: The Frobenius inner product, $S:T=\sum_{i,j=1}^kS_{ij}T_{ij}=\tr(S^*T)$.
    \item $G(n,k)$: The space of k-dimensional linear subspaces in $\RR^n$. We identify an element of $G(n,k)$ with the corresponding orthogonal projection.
    \item $\Tan(V,x),T_xV,V\in\iv_k(U),x\in U$: The approximate tangent space of $V$ at $x$. We will use $T_xV$ for the subspace and $\Tan(V,x)$ for the corresponding orthogonal projection.
    \item $\tr_S,\det_S$: The trace and determinant of a bilinear form $B$ or linear map $A$ over a subspace $S$. In the latter case $\tr_S(A)=S:A$.
    \item $||V||,||\mu||$ with $V\in\iv_k(U),\mu\in\mathcal{RM}(U,\RR^m)$: The total variation measure induced by the varifold $V$ or $\RR^m$-valued measure $\mu$.
    \item $\theta^V$ with $V\in\iv_k(U)$: The multiplicity function of $V$.
    \item $D^Vg,V\in\iv_k(U),g:U\rightarrow \RR^N$: The tangential differential of $g$.
    \item $[DF(x)]^{-1}$ with $F$ a coordinate system for $M$: The inverse of $DF(x)$ as seen as a map to $T_{F(x)}M$.
\end{itemize}
\subsection{Measure Theory and Distributions}

\begin{definition}

    A Radon measure $\mu\in\mathcal{RM}(U)$ is called infinitesimally doubling at $x\in U$ provided that for any $c>1$ there exists $\gamma_c\geq1$ such that $$\limsup_{r\downarrow0}\frac{\mu(B_{cr}(x))}{\mu(B_r(x))}=\gamma_c.$$
\end{definition}

\begin{remark}
 Here, our notion of infinitesimally doubling is different from the usual requirement, where only $\limsup_{r\downarrow0}\frac{\mu(B_{2r}(x))}{\mu(B_r(x))}=\gamma$ is required to hold. Our definition will give better information on the constants $\gamma_c$, and both definitions are equivalent.
\end{remark}

\begin{lemma}
\label{lem:strdbl}
    Suppose that there exists a positive integer $N$ such that $0<\Theta^N(\mu,x)<\infty$. Then $\mu$ is infinitesimally doubling at $x$. Moreover, $\lim_{c\downarrow 1}\gamma_c=1$.
\end{lemma}
\begin{proof}
    We have that \begin{align*}
        1\leq\limsup_{r\downarrow0}\frac{\mu(B_{cr}(x))}{\mu(B_r(x))}=\limsup_{r\downarrow0}\frac{\mu(B_{cr}(x))}{\mu(B_r(x))}\frac{r^N}{r^N}\leq\limsup_{r\downarrow0}\frac{r^N}{\mu(B_r(x))}\limsup_{r\downarrow0}\frac{\mu(B_{cr}(x))}{r^N}=c^N.
    \end{align*}
    The second statement follows immediately.
\end{proof}

We start with the following rather trivial facts, that are likely folklore, but will be useful later.
\begin{lemma}
\label{lem:ASMimpliesDEN}
    Let $U$ be an open subset of $\RR^n$ and let $\nu,\mu$ be Radon measures on $U$. Suppose $x\in U$ is such that $\Theta^n(\nu,x),\Theta^n(\mu,x)$ both exist, and furthermore for any $g\in C^{\infty}_c(\RR^n)$ with $\spt(g)\supseteq B_{r_1}(0),0\leq g$ for $r_1>0$, $$\lim_{r\rightarrow0}\frac{1}{r^n}[\nu(g((*-x)/r))-\mu(g((*-x)/r))]=0.$$
    
    Then $\Theta^n(\nu,x)=\Theta^n(\mu,x)$.
\end{lemma}
\begin{proof}

Let $l>1$ and $g$ be a smooth function such that $0\leq g\leq1$, $g=1$ on $B_1(0)$, and $g=0$ outside of $B_l(0)$. Let $r_0>0$ be such that $B_{lr_0}(x)\subseteq U$.

We have that since the two measures are positive, the following two inequalities hold for $r<r_0$:
$$\nu(B_r(x))-\mu(g((*-x)/r))\leq \nu(B_r(x))-\mu(B_r(x))\leq\nu(g((*-x)/r)) -\mu(B_r(x)).$$

Multiplying by $\omega_n^{-1}r^{-n}$ and applying $\liminf$ to the left hand term, we find that 
\begin{align*}
    \omega_n^{-1}\liminf_{r\downarrow0}r^{-n}[\nu(B_r(x))-\mu(g((*-x)/r))]&=\omega_n^{-1}\liminf_{r\downarrow0}r^{-n}[\nu(B_r(x))-\nu(g((*-x)/r))]\\&+\omega_n^{-1}\liminf_{r\downarrow0}r^{-n}[\nu(g((*-x)/r))-\mu(g((*-x)/r))].
\end{align*}

The second term goes to $0$ by hypothesis, while the first term has $$\omega_n^{-1}\liminf_{r\downarrow0}r^{-n}[\nu(B_r(x))-\nu(g((*-x)/r))]\geq\omega_n^{-1}\liminf_{r\downarrow0}r^{-n}[\nu(B_r(x))-\nu(B_{lr}(x))]\rightarrow(1-l^n)\Theta^n(\nu,x).$$

Similarly, we have that $$\omega_n^{-1}\limsup_{r\downarrow0}r^{-n}[\nu(g((*-x)/r)) -\mu(B_r(x))]\leq(l^n-1)\Theta^n(\mu,x).$$

Therefore we have that for any $l>1$, $$(1-l^n)\Theta^n(\nu,x)\leq \Theta^n(\nu,x)-\Theta^n(\mu,x)\leq(l^n-1)\Theta^n(\mu,x).$$

By taking $l\downarrow1$ the conclusion follows.
\end{proof}

More generally, if $\mu_0$ is a Radon measure and everything is as above, except instead we require $\Theta^{\mu_0}(\mu,x),\Theta^{\mu_0}(\nu,x)$ to exist, and for $\mu_0$ to be infinitesimally doubling at $x$, with $\lim_{c\downarrow 1}\gamma_c=1$, then $\Theta^{\mu_0}(\mu,x)=\Theta^{\mu_0}(\nu,x)$. Here $\Theta^{\mu_0}(\mu,*)$ is the density of $\mu$ with respect to $\mu_0$.

\begin{proposition}
\label{prop:strASMimpliesDEN}
    If everything is as in the preceding paragraph, except that $\mu,\nu$ are signed Radon measures, and moreover the densities $\Theta^{\mu_0}(||\mu||,x),\Theta^{\mu_0}(||\nu||,x)$ exist and are finite, then the same statement holds.

    If both $\mu,\nu$ are $\RR^m$-valued Radon measures and moreover the densities $\Theta^{\mu_0}(||\mu||,x),\Theta^{\mu_0}(||\nu||,x)$ exist and are finite, then the same statement holds.
\end{proposition}
\begin{proof}
    By the Lebesgue decomposition theorem \cite[Theorem~1.31]{Evans_2015} there exists a set $A$ of $\mu_0$ measure 0 such that $\Theta^{\mu_0}(\mu,x),\Theta^{\mu_0}(\nu,x),\Theta^{\mu_0}(||\nu||+||\mu||,x),$ all exist at all $x\not\in A$. 
    
    Thus replacing $\nu\mapsto\nu+||\nu||+||\mu||$ and $\mu\mapsto\mu+||\nu||+||\mu||$ we see the conditions of lemma \ref{lem:ASMimpliesDEN} are met, hence \begin{align*}
        \Theta^{\mu_0}(\mu,x)+\Theta^{\mu_0}(||\nu||+||\mu||,x)=&\Theta^{\mu_0}(\mu+||\nu||+||\mu||,x)=\Theta^{\mu_0}(\nu+||\nu||+||\mu||,x)
        \\
        =&\Theta^{\mu_0}(\nu,x)+\Theta^{\mu_0}(||\nu||+||\mu||,x)
    \end{align*}
    as required.

    To get the second part, use the measures $f\mapsto \int fe_i\dv\mu$ and $f\mapsto \int fe_i\dv\nu$ 
\end{proof}

Given that for any $f\in C^\infty_c(\RR^n,\RR^m)$ the open set $\{x\in V:f(x)\neq 0\}$ is relatively compact in $V$, the next theorem follows immediately by \cite[Chapter~1,~Theorems~5.12~and~5.14]{simon2014introduction}:

\begin{lemma}
\label{lem:disttomeas}
    Let $T$ be an $\RR^m$-valued distribution on an open set $V$, and suppose that for each $U\Subset V$ open, we have a sequence $T_i$ of $\RR^m$-valued distributions converging to $T|_{\mathcal D(U,\RR^m)}$, and such that there exists a Radon measure $\mu$ and a non-negative function $g\in L^1_{loc}(\mu)$ depending only on $T$ with $|T_i(f)|\leq\int g|f|\dv\mu$ for all $i$, and for all fixed $f\in \mathcal D(U,\RR^m)$. 
    
    Then $T$ is an $\RR^m$-valued Radon measure on $V$, $||T||=F\dv\mu$ for some $F\in L^1_{loc}(\mu,\RR^m)$ and $|F|\leq f$ $\mu$-a.e. 

    If, instead, $m=1$ and $T_i(f)\geq-\int |f|g\dv\mu$ for any such sequence as above then $T=F\dv\mu+\nu$ for some $F\in L^1_{loc}(\mu)$, and $\nu$ is a positive Radon measure.
\end{lemma}

\subsection{Optimal Transport}

Let $\mu,\nu$ be Radon probability measures on $\RR^n$ and let $c:\RR^n\times \RR^n\rightarrow\RR$ be continuous. Let $\pi_i:\RR^n\times \RR^n\rightarrow\RR^n,i=1,2$ be the natural projections.

The problem of optimal transportation in the Monge-Kantorovich formulation \cite[p.~10]{Villani_2009} is to solve $P(c,\nu,\mu)=\inf_{\gamma}\int c(x,y)\dv\gamma(x,y)$, where the minimum is taken over all probability measures $\gamma$ such that $\pi_{1,*}\gamma=\nu,\pi_{2,*}\gamma=\mu$.

One advantage of this approach is the following:

Under some mild assumptions of the cost function, this is equal to \cite[equation~5.3]{Villani_2009} $$D(c,\nu,\mu)=\sup\{\int \tilde f\dv\mu-\int\tilde g\dv\nu:\tilde f(x)-\tilde g(y)\leq c(x,y),(\tilde f,\tilde g)\in L^1(\mu\times\nu)\}.$$

In particular, for the distance squared cost function $c(x,y)=-x\cdot y$, and considering $X=\spt(\nu),Y=\spt(\mu)$, we have \begin{theorem}[{\cite[Theorem~5.10(iii)]{Villani_2009}}]
\label{thm:optexist}
    If $c(x,y)=-x\cdot y$ and the supports of $\mu,\nu$ are compact then:
    \begin{enumerate}
        \item $P(c,\nu,\mu),D(c,\nu,\mu)$ both exist and are finite.
        \item $P(c,\nu,\mu)=D(c,\nu,\mu)$
        \item $D(c,\nu,\mu)$ is achieved by a pair $(\tilde f,\tilde g)$ such that $\tilde f(x)=\inf_{y\in \spt(\nu)}(\tilde g(y)-x\cdot y)$. 
    \end{enumerate}
\end{theorem}

We now associate a pair $(f,g)=(- \tilde f, -\tilde g)$, following \cite{Brendle_2023}. Then $f(x)=\sup_{y\in \spt(\nu)}(g(y)+x\cdot y)$ is the pointwise supremum of convex functions, hence is convex.

One more thing to note is that if $f$ is given as above, then for $x_1,x_2\in\spt(\mu)$, \begin{align*}
    |f(x_1)-f(x_2)|&=|\sup_{y_1\in \spt(\nu)}(g(y_1)+x_1\cdot y_1)-\sup_{y_2\in \spt(\nu)}(g(y_2)+x_2\cdot y_2)|
    \\
    &\leq|\sup_{y\in \spt(\nu)}(g(y)+x_1\cdot y)-(g(y)+x_2\cdot y)|
    \\
    &\leq|x_1-x_2|\sup_{y\in \spt(\nu)}|y|.
\end{align*}

Therefore, in the above theorem, $f$ is the restriction of a Lipschitz and convex function on $\RR^n$ with Lipschitz constant \begin{equation}
\label{eqn:Lipopt}
    \lip(f)\leq\sup_{y\in \spt(\nu)}|y|.
\end{equation}
See also the discussion in \cite[p.~3]{Brendle_2023}. We will always identify $f$ with this extension, so that by definition $f(x)-g(y)-x\cdot y\geq0$ for all $x\in\RR^n,y\in\spt(\nu)$.

For such an $\tilde f$, the $c$-superdifferential \cite[Definition~5.7]{Villani_2009} is well-defined. Hence, for our $f$ the $c$-subdifferential at a point $y\in\spt(\mu)$ is well-defined and is exactly $$\partial_cf(y)=\{x\in\spt(\nu):f(y)-g(x)-x\cdot y=0\}.$$

The full $c$-subdifferential is $\partial_cf=\{(x,y)\in\spt(\nu)\times\spt(\mu):f(y)-g(x)-x\cdot y=0\}$

Recall, to keep consistent with Villani, that $c(x,y)=-x\cdot y$, and note this is defined on the target domain, hence differs from the usual definition of the $c$-subdifferential \cite[Definition~5.2]{Villani_2009}.

\section{Geometry of $C^2$-Rectifiable Sets and Approximate Derivatives}
\label{section:C2Geo}

Recall from \cite{anzellotti1994k} that a set $A\subseteq\RR^{n}$ is called countably $(\HH^k,k)C^2$-rectifiable provided that it is $\HH^k$-measurable, $\HH^k(A)<\infty$, and there exists countably many $C^2$-submanifolds $M_i\subseteq\RR^n$, such that $\HH^k(A-\bigcup_{i=1}^\infty M_i)=0$. 

In the sequel a $C^2$ rectifiable $n$-dimensional set (resp. $n$-varifold $V$) will mean a countably $(\HH^n,n)C^2$-rectifiable set (resp. a rectifiable varifold with $\spt(||V||)$ countably $(\HH^n,n)C^2$-rectifiable). By a decomposition of $\spt(||V||),V\in\iv_k(\RR^N)$ we will mean the collection $\{M_i\}_{i=1}^\infty$.

We shall have need of several different types of Taylor expansions.

To do so, we will recall that $P$ is a polynomial of degree at most $k$ on a subspace $S\subseteq\RR^m$, $k\geq0$ provided there are multi-linear maps $P_i:\bigodot^iS\rightarrow\RR^n,i=0,...,k$. It is homogeneous (of degree $j$) provided that $P_i=0,i\neq j$.

We will often use the shorthand $P=\sum_{i=0}^kP_i$ even though this is not defined on the same space, and we will use $P_i(p)=P_i(p,...,p)$ for $p\in S$.

The first expansion is a pointwise Taylor expansion that occurs in the theory of convex functions.

\begin{definition}
    Let $f:A\rightarrow\RR^n$ be defined on an open set $A\subseteq\RR^n$, let $m>0$ be an integer. Then $f$ has a Taylor expansion of order $m$ at $x\in A$ along a subspace $S\subseteq\RR^n$ provided there are multi-linear maps $P_i:\bigodot^iS\rightarrow\RR^n,i=0,...,m$ such that $$||f(x+p)-\sum_{i=0}^mP_i(p)||=o(||p||^m),p\in S.$$

\end{definition}

    More generally, suppose that $S\subseteq\RR^n$ is a subspace, and $Q_i:\bigodot^iS\rightarrow\RR^n,i=0,...,N$ are multi-linear, with $rank(Q_1)=\dim(S)$ and $Q_0=x$, and set $Q(p)=\sum_{i=0}^NQ_i(p)$, then we can say $f$ has a Taylor expansion of order $m$ at $x\in A$ with respect to $Q$ if there are multi-linear maps as above with $||f(Q(p))-\sum_{i=0}^mP_i(p)||=o(||p||^m)$.

\begin{definition}
    Let $f:A\rightarrow\RR^n$ be defined on a closed set $A$, let $m>0$ be an integer. Then $f$ has an $(\HH^k,k)$approximate Taylor expansion of order $m$ at $x\in A$ provided there are multi-linear maps $P_i:\bigodot^iT_xA\rightarrow\RR^n,i=0,...,m$ such that set $$\{y\in A:\frac{||f(y)-\sum_{i=0}^mP_i(y-x)||}{||x-y||^m}\geq\varepsilon\}$$has $\HH^k\Lcorner A$-density 0 at $x$, for all $\varepsilon>0$.
\end{definition}

Let $M$ be a $C^k$ submanifold, $k\geq 2$, of Euclidean space, let $y\in M$, and let $\xi(x)$ be a local $C^k$ retraction to $M$ near $y$. Then $\xi(x)$ possesses a Taylor expansion $P_i$ of order $k$ at each $x\in M$ with respect to the subspace $T_xM$. If $x\in M$, $\xi(x)=P_0=x$, by \cite[Section~2.12.3,~Theorem~1]{Simon_1996}, we know that $P_1=\Tan(M,x)$ has rank $\dim(M)$, and $P_2=\sff_x$.

We may then define the order $k$ approximation of $M$ at $x$ by $Q_{x,M,k}=\sum_{i=0}^kP_i$, where the $P_i$ are as in the last paragraph.

It follows by the appendices that a locally convex function has a Taylor expansion of order 2 with respect to $Q_{x,M,2}$ at $\HH^{\dim(M)}$-a.e. $x\in M$.

Following the ideas in \cite[3.3,~3.11(4)]{Menne_2019} and \cite[2253]{menne2012sharp}, and taking inspiration from \cite[Theorem~19.11(1)]{menne2024sharplowerboundmean} we define:
\begin{definition}
    A set $A$ is weakly $k$ times differentiable at $x\in A$ provided that there is a subspace $S_x\in G(n,m),n\geq m\geq 0$ and a polynomial $P_x:S_x\rightarrow S_x^\perp$ of degree $\leq k$ such that if we set $B=\{v+P_x(v):v\in S\}$ then $x\in B$ and $$\lim_{r\downarrow0}r^{-k}\sup_{x\in B_r(x)\cap A}dist(x,B)=0.$$
\end{definition}

Given a $C^2$-rectifiable set $A$ of finite  $\HH^k$ measure, and a decomposition into $C^2$ submanifolds as above, it holds that by \cite[2.10.19(4)]{Federer_1996}, the function $x\mapsto\Tan(A,x)$ is approximately differentiable $\HH^k$-a.e. in $A$, and we define $\sff_{A,x}$ to be the approximate differential. Moreover, at $\HH^k$-a.e. $x\in A$ there is a quadratic "approximation" of the set using $Q_{x,A,2}\coloneqq Q_{x,M_i,2}=x+\Tan(A,x)+\sff_{A,x}$ on the subspace $T_xA$, where $\Theta^k(\HH^k\Lcorner (A-M_i),x)=\Theta^k(\HH^k\Lcorner (M_i-A),x)=0$.

We may do the same thing with a $C^2$ rectifiable varifold $V$, by taking $A=\spt(||V||)$.

We next quantify how good of an approximation this is:
\begin{definition}
    The quadratic height excess of a $C^2$ rectifiable varifold $V$ at scale $\rho>0$ and power $q\geq 1$ is defined as $$\QHE_\rho(V,x,q)\coloneqq \rho^{-2}\left(\rho^{-n}\int_{B_\rho(x)}dist^q(y,Q_{x,V,2}(T_xV))\dv||V||(y)\right)^{1/q}.$$
\end{definition}

Connecting this back to the notion of pointwise differentiability, we get \begin{definition} \label{def:twicediff}
    A $C^2$ rectifiable varifold is weakly twice differentiable provided that at $||V||$-a.e. $x\in\RR^{n+m}-\spt(||\delta V||_{sing})$, $\spt(||V||)$ is twice differentiable at $x$, and moreover we may take $S=T_xV$ and $P=x+\sff_{V,x}$.
\end{definition}

\begin{proposition}
\label{prop:WTDimpliesQHD}
    Let $V$ be weakly twice differentiable. Then at $||V||$-a.e. $x\in\RR^{n+m}-\spt(||\delta V||_{sing})$ it holds that $\lim_{\rho\rightarrow0}\QHE_\rho(V,x,q)=0$ for all $q\geq1$.
\end{proposition}
\begin{proof}
    We have that by definition if $N=Q_{x,V,2}(T_xV)$ then $\lim_{r\downarrow0}r^{-2}\sup_{x\in B_r(x)\cap A}dist(x,B)=0$ at $||V||$-a.e. $x\in\RR^{n+m}-\spt(||\delta V||_{sing})$. Additionally, we have that $\Theta^n(||V||,x)\in\RR$ at $||V||$-a.e. $x$.

    At a point $x$ satisfying these two conditions we have that \begin{align*}
        \lim_{\rho\rightarrow0}\QHE_\rho(V,x,q)&=\lim_{\rho\rightarrow0}\rho^{-2}\left(\rho^{-n}\int_{B_\rho(x)}dist^q(y,Q_{x,V,2}(T_xV))\dv||V||(y)\right)^{1/q}
        \\
        &\leq\lim_{\rho\rightarrow0}\rho^{-2}\left([\sup_{x\in B_r(x)\cap A}dist(x,B)]^q\rho^{-n}\int_{B_\rho(x)}\dv||V||(y)\right)^{1/q}
        \\
        &=\left[\lim_{\rho\rightarrow0}\rho^{-2}\sup_{x\in B_r(x)\cap A}dist(x,B)\right]\left(\lim_{\rho\rightarrow0}\rho^{-n}\int_{B_\rho(x)}\dv||V||(y)\right)^{1/q}
        \\
        &=(\omega_n\Theta^n(||V||,x))^{1/q}\left[\lim_{\rho\rightarrow0}\rho^{-2}\sup_{x\in B_r(x)\cap A}dist(x,B)\right]=0
    \end{align*}
\end{proof}

We will also need the $L^q$ height-excess:
$$\HE_\rho(V,x,q)\coloneqq \rho^{-1}\left(\rho^{-n}\int_{B_\rho(x)}dist^q(y,x+T_xV)\dv||V||(y)\right)^{1/q}.$$

\begin{definition}
    Let $S\in G(n,m)$, the flat divergence of $g\in C^1(\RR^n,\RR^n)$ with respect to $S$ is $div_S(g)(x)=\tr_S(Dg(x))$. Similarly, given $f\in C^2(\RR^n)$, we define the flat Laplacian with respect to $S$ as $\lap_Sf(x)=div_S(Df)(x)=\tr_S(D^2f(x))$.
\end{definition}

It is clear that if $Dg(x)$ is positive semi-definite, for example if $g=Df$ where $f$ is $C^2$ and convex, then $div_S(g)(x)\geq0$.

Finally, we recall that by \cite{Menne2013Second} (cf. also \cite{santilli2021second} for when the mean curvature is bounded and no singular part) that if $V$ is an integral varifold with locally bounded first variation, then $V$ is $C^2$ rectifiable, and there exists a decomposition $\{M_i\}$ as in the definition, such that for any $M_i$ and $||V||$-a.e. $x\in M_i$, $H(V,x)=H(M_i,x)$( cf. also \cite[Corollaries~4.2~and~4.3]{Sch_tzle_2009} for the case of $L^2_{loc}$ mean curvature and no singular part).

In such a case, by the results of the appendices, for any locally convex function $g$, the Alexandrov Hessian $\hess^{M_i}(g)$ exists at $x$, and for $||V||$-a.e. $x\in M_i\cap M_j$ one has $\hess^{M_j}(g)(x)=\hess^{M_i}(g)(x)$. Hence defining $\hess^{V}(g)(x)=\hess^{M_i}(g)(x),x\in M_i$ yields a bilinear form well-defined a.e.. 

Moreover, there is a second order Taylor expansion of $g$ with respect to $Q_{x,\spt(||V||),2}=Q_{x,M_i,2}$.

\section{Approximations and $L^p$ Taylor Expansions}

In this section we will let $\mu$ be a Radon measure in an open set $U\subseteq \RR^L$.

\begin{definition}
\label{Def:Tk}
    Let $x\in U$, in analogy with \cite{Calder_n_1961}, we say a function $f:U\rightarrow\RR^d$ is in $T^{k,p,d,\mu}(x)$ if there exist multi-linear maps $P_i:\bigodot^i\RR^L\rightarrow\RR^d,i=0,...,k-1$ such that $\limsup_{r\downarrow0}r^{-k}\left( \fint_{B_r(x)} |f-\sum_{i=0}^{k-1}P_i|^p \dv\mu\right)^{1/p}<\infty$. It is in $t^{k,p,d,\mu}(x)$ if it is in $T^{k,p,d,\mu}(x)$ and there exists a multi-linear map $P_k:\bigodot^k\RR^L\rightarrow\RR^d$ such that $\lim_{r\downarrow0}r^{-k}\left( \fint_{B_r(x)} |f-\sum_{i=0}^{k}P_i|^p \dv\mu\right)^{1/p}=0$.
\end{definition}

One reason this definition is useful is in the study of differential operators with $L^p_{loc}(\mu)$ coefficients. More precisely, if $L\in \Diff{\leq k,p}{d,l}(\mu)$, $q=\frac{p}{p-1}$ and $g\in L^q_{loc}(\mu,\RR^l)$, then we may make sense of $L_\mu^*g$ as the $\RR^d$-valued distribution $f\mapsto\int(Lf)\cdot g\dv\mu$.

\begin{lemma}
    Let $g,q,p,\mu,d,k,l$ be as above, with $L\in \Diff{\leq k,p}{d,l}(\mu)$, suppose that $g\in t^{k,q,l,\mu}(x)$ and let $F_x=\sum_{i=0}^{k}P_i$, with $P_i$ as in definition \ref{Def:Tk}. Suppose furthermore that $x$ is a Lebesgue point for the $p$-th power of the coefficients of $L$, and $\mu$ is infinitesimally doubling at $x$.
    
    Then for any $f\in C^\infty_c(U,\RR^d)$ with $\spt(f)\subseteq B_{R}(0)$, $$\lim_{r\rightarrow0}\frac{1}{
\mu(B_r(x))}[(L_\mu^*g)(f((*-x)/r))-(L_\mu^*F_x)(f((*-x)/r))]=0.$$
\end{lemma}
\begin{proof}
    By linearity, we may as well assume that $L=A_\alpha D^\alpha$ where $\alpha$ is a multi-index with $|\alpha|\leq m$, and $A_\alpha\in L^p(\mu,\hom(\RR^k,\RR^l))$.

    For each such $f$ pick $r_0>0$ small enough that $\spt(f((*-x)/r_0))\subseteq U$, and set $C_1=\sup_{y\in \RR^n}|D^\alpha f(y)|<\infty$.

    Then for $r<r_0$ we have $|D^\alpha f((*-x)/r)|\leq \frac{C_1}{r^{|\alpha|}}$.

    Therefore if we set $f_{x,r}=f((*-x)/r)$ then \begin{align*}
        |(L_\mu^*g)(f_{x,r})-(L_\mu^*F_x)(f_{x,r})|&=|\int(A_\alpha D^\alpha f_{x,r})\cdot (g-F_x)\dv\mu|\\
        &\leq \int|A_\alpha| |D^\alpha f_{x,r}||g-F_x|\dv\mu
        \\
        &\leq \frac{C_1}{r^{|\alpha|}}\int_{\spt(f_{x,r})}|A_\alpha||g-F_x|\dv\mu
        \\
        &\leq\frac{C_1}{r^{|\alpha|}}||1_{(\spt(f_{x,r}))}A_\alpha||_{L^p(\mu)}||1_{(\spt(f_{x,r}))}(g-F_x)||_{L^q(\mu)}
        \\
        &\leq\frac{C_1}{r^{|\alpha|}}||1_{B_{rR}(x)}A_\alpha||_{L^p(\mu)}||1_{B_{rR}(x)}(g-F_x)||_{L^q(\mu)}.
    \end{align*}

    On the other hand, by assumption $x$ is a Lebesgue point for $|A_\alpha|^p$ and $$\limsup_{r\downarrow0}r^{-k}\left( \fint_{B_r(x)} |g-F_x|^q \dv\mu\right)^{1/q}=\left(\limsup_{r\downarrow0}r^{-qk} \fint_{B_r(x)} |g-F_x|^q \dv\mu\right)^{1/q}=0.$$

    Therefore, if we take $r>0$ small enough so that $rR<1$, and WLOG take $R>1$, we have \begin{align*}
        0&\leq\limsup_{r\downarrow0}\frac{1}{[\mu(B_r(x))]^{1/q}}r^{-|\alpha|}||1_{B_{rR}(x)}(g-F_x)||_{L^q(\mu)}
        \\
        &\leq\left(\limsup_{r\downarrow0}r^{-qk} \frac{\mu(B_{rR}(x))}{\mu(B_r(x))}\fint_{B_{rR}(x)} |g-F_x|^q \dv\mu\right)^{1/q}
        \\
        &\leq R^k\left( \limsup_{r\downarrow0}\frac{\mu(B_{rR}(x))}{\mu(B_r(x))}\right)^{1/q}\left(\limsup_{r\downarrow0}(rR)^{-qk} \fint_{B_{rR}(x)} |g-F_x|^q \dv\mu\right)^{1/q}
        \\
        &=R^k\gamma_R^{1/q}\left(\limsup_{r\downarrow0}r^{-qk} \fint_{B_{r}(x)} |g-F_x|^q \dv\mu\right)^{1/q}=0.
    \end{align*}

    Also, through a similar computation we get that $$0\leq\limsup_{r\downarrow0}\frac{1}{[\mu(B_r(x))]^{1/p}}||1_{B_{rR}(x)}A_\alpha||_{L^p(\mu)}<\infty$$
    and hence the conclusion follows, since $1=1/p+1/q$.
\end{proof}

From this we get the following:
\begin{theorem}
\label{thm:byparts}
    Suppose $g,q,p,L,\mu,k,l,m$ are as above, $W\subseteq U$ is open, $g\in t^{k,p,l,\mu}(x)$ for $\mu$-a.e. $x\in W$, let $F_x$ be as in the prior lemma, and suppose that there exists a differential operator $\tilde L\in \Diff{\leq k,1,loc}{l,k}$ such that $L_\mu^*(h)(f)=\tilde L_\mu^*(f)(h)=\int (\tilde L h)\dv\mu$ whenever $f,h$ are smooth with compact support, with at least one of them having support contained in $W$, and that the conclusions of lemma \ref{lem:strdbl} hold at $x$
    
    Then, if $L_\mu^*g\Lcorner W$ is an $\RR^m$-valued Radon measure, the absolutely continuous part of $L_\mu^*g\Lcorner W$ is equal to $\tilde LF_x(x)$ at $\mu$-a.e. $x\in W$. In particular, $x\mapsto \tilde LF_x(x)$ is in $L^1_{loc}(\mu\Lcorner W,\RR^m)$ and the following integration-by-parts formula holds: \begin{equation}
    \label{eqn:byparts}
        \int g\cdot Lf\dv\mu=\int\tilde LF_x(x)\cdot f(x)\dv\mu(x)+[L_\mu^*g\Lcorner W]_{sing}(f),~\forall f\in C^\infty_c(W,\RR^l).
    \end{equation}
\end{theorem}
\begin{proof}
    Notice that by hypothesis for any smooth, compact supported $h$, the measure $L_\mu^*h\Lcorner W=\tilde Lh\dv\mu\Lcorner W$.

    Hence, by combining this with the prior lemma, the following properties are true for $\mu$-a.e. $x\in W$:
    \begin{enumerate}
        \item For any smooth function $G$ the $\mu$-density of $L_\mu^*G\Lcorner W$ exists at $x$ and is equal to $\tilde L G(x)$. Moreover, the density of the total variation exists at $x$ and is equal to $|\tilde L G(x)|$.
        \item The densities of $L_\mu^*g\Lcorner W$ and $||L_\mu^*g||\Lcorner W$ exist at $x$.
        \item For any smooth $f\in\mathcal D(W,\RR^m)$, $$\lim_{r\rightarrow0}\frac{1}{
\mu(B_r(x))}[(L_\mu^*g)(f((*-x)/r))-(L_\mu^*F_x)(f((*-x)/r))]=0.$$
    \end{enumerate}

    The statement then follows by proposition \ref{prop:strASMimpliesDEN}.
\end{proof}

In case $L,\tilde L$ are related as in the above theorem, we will often write $\tilde Lg(x)$ instead of $\tilde L F_x(x)$.

In preparation for our result on convex functions we prove a statement for Lipschitz vector fields:
\begin{lemma}
\label{lem:L2TaylorLip}
     Suppose $V\in \iv_n(U)$ is such that $T_xV$ exists, $\lim_{r\downarrow0}\HE_r(V,x,2)=0$, $g:U\rightarrow\RR^{D}$ is Lipschitz, and each component $g_i$ is tangentially differentiable with respect to $V$ at $x$. 
     
     Then $g\in t^{1,2,D,||V||}(x)$ and moreover, $F_x(y)=g(x)+D^Vg(x) (y-x)$ works.
     
\end{lemma}
\begin{proof}
    We have the following two statements true at $x$:
    \begin{enumerate}
        \item $\limsup_{r\downarrow0}r^{-1}\left(\fint_{B^r(x)}dist^2(y,x+T_xV)\dv||V||(y)\right)^{1/2}=0$
        \item $g(x+v)=g(x)+D^Vg(x)v+o(|v|),v\in T_xV$
    \end{enumerate}
    Therefore, since $D^Vg(x)(y-x)=D^Vg(x)\Tan(V,x)(y-x)$, we have that \begin{align*}
        g(y)-(g(x)+D^Vg(x)(y-x))&=g(y)-g(x+\Tan(V,x)(y-x))
        \\
        &+g(x+\Tan(V,x)(y-x))-(g(x)+D^Vg(x)\Tan(V,x)(y-x)).
    \end{align*}

    Now, we have that $|g(y)-g(x+\Tan(V,x)(y-x))|^2\leq|dist(y,x+T_xV)|^2$ and $$|g(x+\Tan(V,x)(y-x))-(g(x)+D^Vg(x)\Tan(V,x)(y-x))|^2=o(|\Tan(V,x)(y-x)|^2).$$ Since $\Tan(V,x)$ has norm at most 1 as a linear map, we then also have $$|g(x+\Tan(V,x)(y-x))-(g(x)+D^Vg(x)\Tan(V,x)(y-x))|^2=o(|y-x|^2).$$

    Since $T_xV$ exists we have that the density exists at $x$ and is finite.

    Therefore, we conclude that \begin{align*}
        &\limsup_{r\downarrow0}r^{-2}\fint_{B^r(x)}|g(y)-(g(x)+D^Vg(x)(y-x))|^2\dv||V||(y)
        \\
        &\leq \limsup_{r\downarrow0}r^{-2}\fint_{B^r(x)}|dist(y,x+T_xV)|^2\dv||V||(y)+r^{-2}o(r^2)\fint_{B^r(x)}\dv||V||(y)
        \\
        &=\limsup_{r\downarrow0}(o(1)+o(1))=0
    \end{align*}

    Hence we have that \begin{align*}
        &\lim_{r\downarrow0}r^{-1}\left(\fint_{B^r(x)}|g(x+y)-(g(x)+D^Vg(x)(y-x))|^2\dv||V||(y)\right)^{1/2}
        \\
        &=\left(\lim_{r\downarrow0}r^{-2}\fint_{B^r(x)}|g(x+y)-(g(x)+D^Vg(x)(y-x))|^2\dv||V||(y)\right)^{1/2}=0
    \end{align*}
\end{proof}

\begin{theorem}
        Suppose $V\in \iv_n(U)$, $W\subseteq U$ is open, and at $||V||$-a.e. $x\in W\subseteq U$: $T_xV$ exists, $\lim_{r\downarrow0}\QHE_r(V,x,2)=0$, $g:U\rightarrow\RR$ is locally convex and $\hess^V(g)(x)$ exists. 
     
        Then $g\in t^{2,2,1,||V||}(x)$ at $||V||$-a.e. $x\in W$, and moreover, $F_x(y)=g(x)+\grad^Vg(x)\cdot (y-x)+\hess^V(g)(x)(y-x,y-x)$ works.
\end{theorem}
\begin{proof}
    Notice that, if we define $Q$ as the smooth manifold generated by $v\mapsto x+v+\sff_x(v,v)$, then similarly to the prior lemma we have the following two properties at $x$:
    \begin{enumerate}
        \item $\limsup_{r\downarrow0}r^{-2}\left(\fint_{B^r(x)}dist^2(y,Q)\dv||V||(y)\right)^{1/2}=0$
        \item $g(y)=g(x)+\grad^Vg(x)\cdot (y-x)+\hess^V(g)(x)(y-x,y-x)+o(|x-y|^2),y\in Q$
    \end{enumerate}

    Let $P$ be a local, Lipschitz, and $C^2$ retraction to $Q$ near $x$ and $$u(y)=g(x)+\grad^Vg(x)\cdot (y-x)+\hess^V(g)(x)(y-x,y-x).$$ Notice also that locally convex functions are locally Lipschitz, so if we take $r$ sufficiently small and $y\in B_r(x)$ then we may compute that $$|g(y)-g(P(y))|^2\leq C|dist(y,Q)|^2$$ and $$|g(P(y))-u(P(y))|^2=o(|P(y)|^4).$$

    Now, we have that similarly to before, since we may take the retraction to be Lipschitz $$|g(P(y))-u(P(y))|^2=o(|y|^4).$$

    Finally, we have that similarly to the case with $g$, $$|u(P(y))-u(y)|\leq C|dist(y,Q)|^2.$$

    The conclusion then follows by similar logic to the last lemma.
\end{proof}

As a consequence of this, we have that for any $M_i$ is the decomposition, and for $||V||$-a.e. $x\in M_i$ the gradient and Hessian of a locally convex, L-Lipschitz function exists at $x$ and agree with the corresponding gradient and Hessian of $M_i$.

\section{Integration-by-parts Formulae}
\subsection{Integration-by-Parts for Smooth Functions}
In this subsection we will suppose $V\in\iv_n(U)$ has locally bounded first variation, with $U\subseteq\RR^{n+m}$ open.

By \cite[Section~4.3]{allard1972first} there exists an $\RR^{n+m}$-valued Radon measure $\sigma$ such that for any vector field $g\in\mathcal D(U,\RR^{n+m})$ the following integration-by-parts holds: $$\int div_{T_xV}(g(x))\intvar{x}=-\int g(x)\cdot H(V,x)\intvar{x}+\int g\cdot\dv\sigma.$$

In our notation $\sigma=(\delta V)_{sing}$, which we will now adopt.

If we let $g=fG$ for $f\in\mathcal D(U,\RR),G\in\mathcal D(U,\RR^{n+m})$, then this reads (cf. also \cite[Section~7.5]{allard1972first}): \begin{equation}
\label{eqn:divthm}
    \int \grad^Vf\cdot G+fdiv_{T_xV}G\intvar{x}=-\int fG\cdot H(V,x)\intvar{x}+\int fG\cdot\dv(\delta V)_{sing}.
\end{equation}

By putting $G=Dh,h\in\mathcal D(U,\RR)$ this becomes $$\int \grad^Vf\cdot \grad^Vh+f\lap_{T_xV}h\intvar{x}=-\int fDh(x)\cdot H(V,x)\intvar{x}+\int fDh\cdot\sigma$$ and subtracting by the corresponding formula for $g=hDf$ we get $$\int f\lap_{V}g-g\lap_{V}f\intvar{x}=\int (fDh-hDf)\cdot\dv(\delta V)_{sing}$$ where we have set $\lap_VF(x)=\lap_{T_xV}F(x)+DF(x)\cdot H(V,x),F\in C^\infty(U)$. We also set $div_V(F)(x)=div_{T_xV}F(x)+F(x)\cdot H(V,x),F\in C^\infty(U,\RR^{n+m})$

\begin{definition}
    The generalized boundary set of $V$ is defined as $\text{bnd} V=\spt(\sigma)(=\spt(||\delta V||_{sing}))$.
\end{definition}

Let $W=U-\text{bnd} V$ and suppose that the support of at least one of $f,h$ is contained in $W$. Then we have $\int f\lap_{V}h-h\lap_{V}f\intvar{x}=0$.

We also have, under the same hypothesis, $\int f\lap_{V}g\intvar{x}=-\int \grad^Vf\cdot\grad^Vh\intvar{x}$.

From here the next lemma is immediate.

\begin{lemma}
\label{lem:bypartsbounds}
    Let $f\in\mathcal D(W),F\in\mathcal D(W,\RR^{n+m})$. Then:
    \begin{enumerate}
        \item If $h:U\rightarrow\RR^{n+m}$ is smooth and L-Lipschitz then there is a dimensional constant $C>0$ such that $$|\int \grad^Vf\cdot h\intvar{x}|\leq \int(CL+|H(V,x)h(x)|)|f|\intvar{x}.$$
        \item If $h:U\rightarrow\RR$ is smooth and L-Lipschitz  then $|\int hdiv_V(F)\intvar{x}|\leq L\int|F|\intvar{x}$.
        \item If $h:U\rightarrow\RR$ is smooth, L-Lipschitz, and locally convex, then $\int \lap_{V}fh\intvar{x}\geq-\int |f|L|H(V,x)|\intvar{x}$.
    \end{enumerate}
\end{lemma}

\subsection{Distributional Divergences and Gradients}

\begin{proposition}
    Suppose $V$ has locally bounded first variation, $g:U\rightarrow\RR^{n+m}$ is L-Lipschitz, and we define $L$ as the differential operator $Lf(x)=-\grad^Vf(x)=-\Tan(V,x)Df(x)\in \Diff{\leq 1,\infty}{1,n+m}(||V||)$. Then $div_V(g)(f)=-\int\grad^Vf\cdot g\dv||V||$ is a signed Radon measure on $W=U-\text{bnd} V$. 
    
    Moreover, $div_V(g)$ is absolutely continuous with respect to $||V||$ and $\dv [div_V(g)](x)=(\tr_{T_xV}(D^Vg(x))+g(x)\cdot H(V,x))\dv||V||(x)$.
\end{proposition}
\begin{proof}
    Let $f\in\mathcal D(W,\RR)$ be fixed. Then for $\varepsilon>0$ take $g_\varepsilon=\rho_\varepsilon*g$ be a mollification of $g$. Since $W\subseteq U$, for $\varepsilon$ small enough $g_\varepsilon$ is well-defined on $\spt(f)$. Moreover, it is also L-Lipschitz and converges uniformly on $\spt(f)$ to $g$.

    Thus if we take a subsequence $\varepsilon_i\downarrow0$ the distributions $$T_i(f)=-\int\grad^Vf\cdot g_{\varepsilon_i}\dv||V||$$ converge to $div_V(g)(f)$.

    By lemma \ref{lem:bypartsbounds}, and by the uniform convergence, there exists $K>0$ such that $$|T_i(f)|\leq\int(CL+K|H(V,x)|+|H(V,x)g(x)|)|f|\intvar{x}.$$

    Thus by lemma \ref{lem:disttomeas} $div_V(g)\Lcorner W=\lim_iT_i$ is a signed measure, which is absolutely continuous with respect to $||V||$.

    On the other hand, by lemma \ref{lem:L2TaylorLip} and equation \ref{eqn:divthm}, in addition to theorem \ref{thm:byparts}, we deduce that the absolutely continuous part of $div_V(g)\Lcorner W$ is equal to $div_V(g(x)+D^Vg(x)(*-x))(x)=\tr_{T_xV}(D^Vg(x))+g(x)\cdot H(V,x)$ at a.e. $x\in V$.

    In particular, by equation \ref{eqn:byparts}, we deduce that for any $f\in\mathcal D(W,\RR)$ the integration-by-parts $$-\int\grad^Vf\cdot g\dv||V||=-\int f[\tr_{T_xV}(D^Vg(x))+g(x)\cdot H(V,x)]\intvar{x}$$ holds.
\end{proof}

Through a similar proof we deduce that \begin{proposition}
    Suppose $V$ has locally bounded first variation, $g:U\rightarrow\RR$ is L-Lipschitz, and we define $L$ as the differential operator $Lf(x)=div_V(f(x))\in \Diff{\leq 1,\infty}{n+m,1}(||V||)$. Then $-\grad^{V,dist}g(f)=\int gdiv_V(f)\dv||V||$ is an $\RR^{n+m}$-valued Radon measure on $W=U-\text{bnd} V$. 
    
    Moreover, $-\grad^{V,dist}(g)$ is absolutely continuous with respect to $||V||$ and $\dv [-\grad^{V,dist}(g)](x)=-\grad^Vg(x)\dv||V||(x)$.
\end{proposition}

Using (3) of lemma \ref{lem:bypartsbounds} we deduce Alexandrov's theorem for twice differentiable varifolds:
\begin{theorem}[Alexandrov's Theorem for Varifolds]
\label{thm:alex}
    Suppose $V$ has locally bounded first variation, is weakly twice differentiable, $H(V,x)\in L^2_{loc}(||V||,\RR^{n+m})$, $g:U\rightarrow\RR$ is L-Lipschitz and locally convex, and we define $L$ as the differential operator $$Lf(x)=div_V(Df(x))=\lap_Vf(x)\in\Diff{\leq 2,2}{1,1}(||V||).$$ Then $\lap_Vg(f)=\int g\lap_Vf\dv||V||$ is a signed Radon measure on $W=U-\text{bnd} V$. 
    
    Moreover, the singular part of $\lap_Vg$ is non-negative and $\Theta^{||V||}(\lap_Vg,x)=\tr_{T_xV}(\hess^V(g)(x))$ at a.e. $x\in W$.

    In particular, for any $f\in\mathcal{D}(W,\RR),f\geq0$ $$\int f\tr_{T_xV}(\hess^V(g)(x))\intvar{x}\leq-\int \grad^Vf\cdot\grad^Vg\intvar{x}$$
\end{theorem}
\begin{proof}
    Let $A\Subset W$, so that for any symmetric mollifier there exists some $\varepsilon_0>0$ such that $g_\varepsilon=\rho_\varepsilon*g$ is well-defined as a function on $A$.

    Then by uniform convergence for any $f\in\mathcal D(W)$ $\lap_Vg_\varepsilon(f)$ is well-defined and $\lap_Vg_\varepsilon\rightarrow \lap_Vg$ in the sense of distributions on $W$ as $\varepsilon\rightarrow 0$.

    Since $g_\varepsilon$ is smooth, locally convex, and also L-Lipschitz, and $f$ is compactly supported away from $\text{bnd} V$, by lemma \ref{lem:bypartsbounds} we have $$\lap_Vg_\varepsilon(f)\geq-\int |f|L|H(V,x)|\intvar{x}.$$

    By taking a sequence $\varepsilon_i\rightarrow0$ we use the latter part of lemma \ref{lem:disttomeas} to determine that $\lap_Vg$ is a signed Radon measure and $[\lap_Vg]_{sing}$ is a (positive) Radon measure.
    
    By \ref{lem:strdbl} and \ref{prop:WTDimpliesQHD} we therefore deduce that at $||V||$-a.e. $x\in W$:
    \begin{enumerate}
        \item The conclusions of lemma \ref{lem:strdbl} hold at $x$.
        \item $T_xV$ exists.
        \item $H(V,x)\in\nor(V,x)$ by \cite[Theorem~5.8]{Brakke_2015}.
        \item $\lim_{\rho\downarrow 0}\QHE_\rho(V,x,2)=0$.
        \item $\hess^V(g)(x)$ exists.
        \item $g\in t^{2,2,1,||V||}(x)$.
        \item $DF_x(x)=\grad^Vg(x)\in\Tan(V,x)$.
    \end{enumerate}

    Therefore, by theorem \ref{thm:byparts} we have that $\lap_Vg=\lap_VF_x(x)\dv||V||(x)+[\lap_Vg]_{sing}$.

    On the other hand, since $DF_x(x)\in\Tan(V,x)$, it holds that $$\lap_VF_x(x)=\tr_{T_xV}(\hess^V(g)(x))+H(V,x)\cdot DF_x(x)=\tr_{T_xV}(\hess^V(g)(x)).$$

    Thus the two statements follow.

    By the previous lemma we deduce that by setting $F=Df$, that $$\lap_Vg(f)=-\grad^{V,dist}g(F)=-\int\grad^Vg(x)\cdot F\intvar{x}=-\int\grad^Vg(x)\cdot \grad^Vf(x)\intvar{x}$$
    while at the same time $$\lap_Vg(f)=[\lap_Vg]_{sing}(f)+\int f\tr_{T_xV}(\hess^V(g)(x))\intvar{x}\geq\int f\tr_{T_xV}(\hess^V(g)(x))\intvar{x}.$$

    This is exactly the third statement.
\end{proof}

\section{Estimating the Intrinsic Laplacian}

This section will be largely similar to the paper of Brendle and Eichmair, see \cite{Brendle_2023}. We include the statements to make this self-contained, but we will primarily refer to the original paper and illustrate the necessary modifications. The modifications are necessary due to this setting not possessing enough regularity to allow the exponential map and parallel transport to have nice properties, and obtaining a slightly more general result about the density of a general Radon measure $\mu$.

We set $M\subseteq \RR^{n+m}$ to be a connected, $C^2$ submanifold, $\mu$ a Radon probability measure with bounded support, and $\nu,\rho,\alpha$ as in \cite[Theorem~1]{Brendle_2023}.

\begin{remark}
    We are forced to use the underlying $C^2$ submanifolds, rather than the smooth quadratic approximations due to the density estimate in proposition \ref{prop:densest}.
\end{remark}

Then by theorem \ref{thm:optexist} there exist functions $f,g:\RR^{n+m}\rightarrow\RR$ with  $$D(c,\nu,\mu)=\int g\dv\nu-\int f\dv\mu.$$

Moreover, by equation \ref{eqn:Lipopt} we may assume that $f$ is convex and $\lip(f)\leq 1$. We also recall the $c$-subdifferential $\partial_cf(x),\partial_cf$.

The proof of \cite[Lemma~3]{Brendle_2023} carries over in the following form:
\begin{lemma}
    Let $E\subseteq\RR^{n+m},G\subseteq\bar{B}^{n+m}$ be compact and such $[(\bar{B}^{n+m}\backslash G)\times E]\cap\partial_cf=\emptyset$. Then $\mu(E)\leq\nu(G)$.
\end{lemma}
\begin{proof}
   We have $E\cap\spt(\mu)$ is a compact subset of $\spt(\mu)$, and $\mu(E)=\mu(E\cap\spt(\mu))$. The rest is exactly the same as the proof of \cite[Lemma~3]{Brendle_2023}.
\end{proof}

To make the conclusions necessary for the definition of $\hat \omega$, in particular remark \ref{rmrk:GradTaylor}, more clear, we use a notion of subdifferential of $M$ being induced by coordinate charts, rather than ambient space.

Let $r=r_x>0$ be such that the Riemannian disk $D_x=\bar B_{r_x}(x)\subseteq M$ is a compact submanifold of $\RR^{n+m}$ with boundary, and is contained in the image of a single set of graph coordinates (see appendix \ref{AppendixB}); then fix $K_x>0$ such that $$\sup_{z\in D_x,y\in T^\perp_xM,|y|\leq1}|(x-z)\cdot y|\leq K_xd_M^2(x,z).$$

To see this is possible, use graph coordinates (cf. appendix \ref{AppendixB}) to see that, setting $\tilde x=F^{-1}(x)$, $z=F(\tilde x)+DF(0)v+O(|v|^2)$. Thus, since all distances are equivalent and $F$ is bilipschitz onto its image near $x$, $x-z=w+O(|x-z|^2),w=DF(0)v$, so $(x-z)\cdot y=O(|x-z|^2)$.

One departure we will make from \cite{Brendle_2023} is the definition of the subdifferential. 

\begin{definition}[{\cite[Definition~8,~Proposition~9]{Azagra_2007}}]
    The subdifferential of $f$ at $x$ is defined as $$\partial_M f(x)=\partial(f\circ F)(0)(DF(0))^{-1}=[D(F^{-1})(0)]^*\partial(f\circ F)(0)$$ where $F$ form graph coordinates at $x$, and $\partial(f\circ F)$ is the usual proximal subdifferential (see \cite[Definition~7.25]{Clarke_2013}).
\end{definition}

Notice that if $y\in \partial_M f(x)$ then $|y|\leq 1$ by a slight modification of \cite[Proposition~11]{Azagra_2007}.

\begin{remark}
    As is common in this paper, we cannot use the exponential map for the subdifferential since the exponential map is only $C^1$, while $C^2$ is required for a satisfactory definition.
\end{remark}

\begin{remark}
    $\partial_M f(x)$ really lives inside of $T^*_xM$, which is to be expected for differential properties, but we identify it with the subset of $T_xM$ given by lowering the index (turning it into a column vector).
\end{remark}

\begin{remark}
\label{rmrk:GradTaylor}
    Since $f\circ F$ is semiconvex with a quadratic modulus of semiconvexity near $0$, this is actually equal to the standard subdifferential for convex functions, see \cite[Section~4.1]{Clarke_2013}. In particular, by combining this with the results in the appendices and \cite[Theorem~14.25]{Villani_2009}, we deduce that, if $\hess^M(f)(x)$ exists, then for $y$ close enough to $x$, with $u$ given in appendix \ref{AppendixB}, $$\partial_M f(y)=\grad^Mu(y)+o(d_M(y,x)).$$
\end{remark}

We also notice that since the graph coordinates locally look like a rotation of $\RR^n\times\{0\}^m$ plus a small perturbation, if $v=DF(0)w,y=F(a)$ then $$v\cdot(y-x)=DF(0)w\cdot DF(0)a+o(d_M^2(x,y))=w\cdot a+o(d_M^2(x,y)).$$ Hence, the following modification of \cite[Corollary~10]{Azagra_2007} holds:

\begin{lemma}
   If $v\in T_xM$, then $v\in\partial_M f(x)$ if and only if there exists $\sigma\geq 0$ and $U\ni x$ such that for all $y\in U\cap M$ $$f(y)-f(x)\geq v\cdot(y-x)-\sigma d_M^2(y,x).$$
\end{lemma}
\begin{proof}
    We recall that $v\in T_xM$ is in $\partial_M f(x)$ iff $v$ is the lowered index of $\tilde v$, where $\tilde v$ is such that $\tilde v=\tilde w(DF(0))^{-1},\tilde w\in\partial(f\circ F)(0)$.

    By definition, $\tilde w\in\partial(f\circ F)(0)$ iff there exists $\sigma\geq 0$ such that $$f(F(a))-f(F(0))\geq \tilde w(a)-\sigma|a|^2,\forall a\text{ near }0.$$

    Notice that if $w$ is the lowered index form of $\tilde w$, then since $(DF(0))^T|_{T_xM}=(DF(0))^{-1}$ we have $DF(0)w=v$. Hence, it follows that if $y=F(a)$ then $\tilde w(a)=w\cdot a=v\cdot(y-x)+o(d_M^2(x,y))$.

    The conclusion follows.
\end{proof}

We restate the next lemma, but the proof is identical to that in the paper of Brendle and Eichmair, using the prior lemma.

\begin{lemma}[{\cite[Lemma~4]{Brendle_2023}}]
    Let $x\in M$, then $\Tan(M,x)\partial_cf(x)\subseteq\partial_Mf(x)$.
\end{lemma}

\begin{remark}
    This is, in essence, just a set-valued and non-smooth version of the well-known fact that "the projection of the differential of a smooth function to the tangent space of a submanifold is equal to the Riemannian gradient of the function on that submanifold."
\end{remark}

We now define $\hat{\omega}_x(r)$ as $$\inf\{\omega:|y-\grad^Mu(\tilde x)|\leq \omega,\forall y\in\partial f(\tilde x),d_M(\tilde x,x)<\min\{r_x,r\}\}.$$

The proof of the following lemma is identical to that of \cite{Brendle_2023}, using remark \ref{rmrk:GradTaylor}.

\begin{lemma}[{\cite[Lemmas~5~and~6]{Brendle_2023}}]
   For each $x\in M$, $\hat{\omega}_x$ and $\hat{\delta}_x$ are non-decreasing and $\hat{\omega}_x(r)=o(r),~\hat{\delta}_x(r)=o(1)$.
\end{lemma}

Here we have put $\hat{\delta}_x(r)$ as the infimum of the set of all $\delta\in\RR$ satisfying $[\hess^V(u)(\hat x)-\sff^M_{\hat x}\cdot\xi]|_{T_{\hat x}M}\geq -\delta I|_{T_{\hat x}M}$, for each $\hat x\in M,\xi\in\bar{B}^{n+m}$ such that $d_M(\hat x,x)<\min\{r_x,r\}$ and $\xi\in\partial_cf(x)$.

Here $\sff^M_x\cdot\xi$ is the bilinear form defined by $(\sff^M_x\cdot\xi)(\eta,\gamma)=[\sff^M_x(\eta,\gamma)]\cdot\xi$.

Let us make a brief digression here to discuss what is going on with $\hat\delta,\hat\omega$. The idea behind what $\hat\omega_x$ is measuring, is precisely the "Taylor expansion" in \ref{rmrk:GradTaylor}. Since $\partial_Mf(\hat x)$ is multi-valued, we want a \textit{single-valued} function which measures the same thing that the Taylor expansion measures. 

On the other hand, following \cite[p.~5]{CASTILLON201079} if $f$ was smooth and $u=f$, we would have that \begin{align*}
    [\hess^V(u)(\hat x)-\sff^M_{\hat x}\cdot\xi]|_{T_{\hat x}M}&=[D^2u(\hat x)+\sff^M_{\hat x}\cdot\grad^Mu(\hat x)-\sff^M_{\hat x}\cdot\grad^Mu(\hat x)]|_{T_{\hat x}M}
    \\&=[D^2u(\hat x)]|_{T_{\hat x}M}\geq 0.
\end{align*}
Since $f\neq u$ and $f$ is not smooth, and there is no single well-defined $\grad^Mf(\hat x)$, we must make due with an approximate version of this, which is represented by $\hat \delta$.

We now consider \cite[Lemma~7,~Lemma~8,~Proposition~9]{Brendle_2023}.

Instead of using the sup-norm to define $W_r,E_r$, we use the standard Euclidean norm. Namely, we set $0<r<r_x$ and take $$W_r=\{y\in T_xM:|y|\leq r\},~E_r=\exp_x(W_r)=\bar B^M_r(x).$$

This will provide a more straightforward way to use the $n$-dimensional densities, to align with standard measure theory conventions (and, of course, the conventions used in prior sections). For clarity, we will also use $\bar B^M_r(x)$ instead of $E_r$ for the closed Riemannian ball.

\begin{remark}
    Since, near $0$, $\exp_x$ is a $C^1$ diffeomorphism onto a neighborhood of $x$ in $M$ by \cite{lange2024regularitygeodesicflowsubmanifolds}, we see that upon possibly shrinking $r_x$, we may assume that $\bar B^M_r(x)$ is compact.
\end{remark}

The definition of $A_r$ remains the same: \begin{align*}
    A_r=\{(\hat x,z):&\hat x\in \bar B^M_r(x),z\in T^\perp_{\hat x}M,|\grad^Mu({\hat x})|^2+|z|^2\leq(1+\omega(r))^2,
\\
&[\hess^V(u)(\hat x)-\sff^M_{\hat x}\cdot\xi]|_{T_{\hat x}M}\geq -\delta I|_{T_{\hat x}M}\}.
\end{align*}

Since $u$ is smooth and $\bar B^M_r(x)$ is compact, this is closed and bounded, thus is also compact.

On the other hand, we define $G_r$ as in \cite[p.~7]{Brendle_2023}, namely take $\Phi(\hat x,z)=\grad^Mu(\hat x)+z,\hat x\in M,z\in T^\perp_{\hat x}M$, and let $$G_r=\bar B^{n+m}\cap\{p:|p-\Phi(A_r)|\leq\omega(r)\}.$$

Notice that since $\Phi$ is continuous and $A_r$ is compact, the set $\{p:|p-\Phi(A_r)|\leq\omega(r)\}$ is closed and bounded, and hence $G_r$ is compact.

We would like to show that for each $0<r<r_x$, $[(\bar{B}^{n+m}\backslash G_r)\times \bar B^M_r(x)]\cap\partial_cf=\emptyset$, and then we may deduce that $\mu(\bar B^M_r(x))\leq\nu(G_r)$. One thing to note here is that, since $\bar B^M_r(x)\subseteq M$, the resulting density will not be that of the measure $\mu$, but rather the restriction $\mu\Lcorner M$.

The following lemma is true:
\begin{lemma}[{\cite[Lemmas~7~and~8]{Brendle_2023}}]
    If $0<r<r_x$ then $[(\bar{B}^{n+m}\backslash G_r)\times \bar B^M_r(x)]\cap\partial_cf=\emptyset$. In particular, $\mu\Lcorner M(\bar B^M_r(x))=\mu(\bar B^M_r(x))\leq\nu(G_r)$.
\end{lemma}
The proof requires no modifications.

Then we must deal with proposition 9. 

\begin{proposition}[{\cite[Proposition~9]{Brendle_2023}}]
\label{prop:densest}
    Let everything be as above.

    Let $$S=\{z\in T^\perp_xM:|\grad^Mu(x)|^2+|z|^2\leq1,[\hess^V(u)(x)-\sff^M_x\cdot z]|_{T_xM}\geq0\}.$$

    Then $$\Theta^n(\mu\Lcorner M,x)\leq \int_S\det_{T_xM}\left(\hess^V(u)(x)-\sff^M_x\cdot z\right)\rho(|\grad^Mu(x)|^2+|z|^2)\dv z$$
\end{proposition}
\begin{proof}
    We cannot generalize the proof immediately, since the parallel transport is not well-defined in this regularity. 

    To find a replacement we change the map $\Psi$ to $$\Psi(\hat x,y)=(\exp_x(\hat x),\nor(M,\exp_x(\hat x))y).$$

    This is a $C^1$ map by \cite{lange2024regularitygeodesicflowsubmanifolds}, and moreover $$D\Psi(x,y)=\begin{pmatrix}
        Id_{T_xM} & 0
        \\
        * & Id_{T^\perp_xM}.
    \end{pmatrix}$$

    The rest of the proof is the same as that of \cite[Proposition~9]{Brendle_2023} up to the density estimate. Since the factor occurring in $M$ is a Riemannian ball instead of a Riemannian cube, we notice that the estimate is now $$\mu\Lcorner M(\bar B^M_r(x))\leq \nu(G_r)\leq \omega_nr^n\int_S\det_{T_xM}\left(\hess^V(u)(x)-\sff^M_x\cdot z\right)\rho(|\grad^Mu(x)|^2+|z|^2)\dv z.$$

    From here, the density estimate follows, since if $\bar B^M_r(x)$ denotes the (closed) Riemannian ball of radius $r$ and center $x$, then:
    \begin{enumerate}
        \item $$\lim_{r\downarrow 0}\frac{\HH^n(\bar B^M_r(x))}{\omega_nr^n}=1.$$ 
        \item If $\hat B_r(x)$ denotes the intersection of $M$ with the \textit{Euclidean} ball of radius $r$ and center $x$, then as $r\downarrow 0$ $$\hat B_{(1-O(1))r}(x)\subseteq \bar B^M_r(x)\subseteq \hat B_{(1+O(1))r}(x).$$
    \end{enumerate}
\end{proof}

We now finally deduce:
\begin{lemma}[{\cite[Corollary~10]{Brendle_2023}}]
    Let everything be as above, and $[\Theta^n(\mu\Lcorner M,x)]^{1/n}=c>0$. Then \begin{equation}
    \label{eqn:first}
        \frac{n}{\alpha^{1/n}}c\leq\tr_{T_xM}(\hess^M(f)(x))+|H(V,x)|.
    \end{equation}
    In particular, \begin{equation}
    \label{eqn:intrhess}
        \frac{n}{\alpha^{1/n}}\Theta^n(\mu\Lcorner M,x)\leq[\tr_{T_xM}(\hess^M(f)(x))+|H(V,x)|][\Theta^n(\mu\Lcorner M,x)]^{(n-1)/n}.
    \end{equation}
\end{lemma}
\begin{proof}
    The proof is very similar to that of \cite[Corollary~10]{Brendle_2023}, but we will include the proof anyways, since there are some minor modifications.

    If not, there exists some $\beta>0$ with $\beta^{-1}\alpha<1$ such that $\tr(\hess^M(f)(x))+|H(M,x)|<c\frac{n}{\beta^{1/n}}$.

    Then by AM-GM $$0\leq\det_{T_xM}\left(\hess^M(u)(x)-\sff^M_x\cdot z\right)\leq\left(\frac{\tr(\hess^M(f)(x))+|H(M,x)|}{n}\right)^n<c^n\beta^{-1}$$

    With $S$ as in the prior proposition we have \begin{align*}
        \Theta^n(\mu\Lcorner M,x)&\leq \int_S\det_{T_xM}\left(\hess^V(u)(x)-\sff^M_x\cdot z\right)\rho(|\grad^Mu(x)|^2+|z|^2)\dv z
        \\
        &\leq \beta^{-1}\alpha c^n <c^n=\Theta^n(\mu\Lcorner M,x).
    \end{align*}
    This is clearly impossible, establishing \ref{eqn:first}.

    To see equation \ref{eqn:intrhess} multiply both sides of \ref{eqn:first} by $[\Theta^n(\mu\Lcorner M,x)]^{(n-1)/n}$.
\end{proof}

\section{Proof of Theorem \ref{Thm:Theorem1}}

It is clear that the theorem is trivially true if $F=0$ identically. If not, then $\int F^{n/(n-1)}\intvar{x}>0$.

The proof in the general case follows by scaling the domain of $||V||,F$ by a suitable $r>0$ to get $\int F^{n/(n-1)}\intvar{x}=1$, and then scaling back, so we will only prove theorem 1 in the case $\int F^{n/(n-1)}\intvar{x}=1$.

\begin{proof}[Proof of Theorem \ref{Thm:Theorem1} if $\int F^{n/(n-1)}\intvar{x}=1$]

We apply the results of the last section with $\dv\mu=F^{n/(n-1)}\intvar{x}$. 

Namely, since $V$ is $C^2$ rectifiable with compact support, apply \ref{eqn:intrhess} to each submanifold $M_i$ in the decomposition of $\spt(||V||)$ (see section \ref{section:C2Geo}) to get that at $||V||$-a.e. $x\in M_i$ $$\frac{n}{\alpha^{1/n}}\Theta^n(\mu\Lcorner M,x)\leq[\tr_{T_xM}(\hess^M(f)(x))+|H(V,x)|][\Theta^n(\mu\Lcorner M,x)]^{(n-1)/n}.$$ 

On the other hand by \cite[2.10.19(4)]{Federer_1996} at $\HH^n$-a.e. $x\in M_i$ $\Theta^n(\mu-M_i,x)=0$, hence at all such $x$ $$F^{n/(n-1)}(x)\theta^V(x)=\Theta^n(\mu,x)=\Theta^n(\mu\Lcorner M,x).$$

Since $V$ is rectifiable, this is also true at $||V||$-a.e. $x\in M_i$. Hence in particular at all such $x$, $$\frac{n}{\alpha^{1/n}}F^{n/(n-1)}(x)\theta^V(x)\leq F(x)[\theta^V(x)]^{(n-1)/n}[\tr_{T_xM}(\hess^M(f)(x))+|H(M,x)|].$$

Additionally, at such a point, $\theta\in\mathbb N$, hence $[\theta^V(x)]^{(n-1)/n}\leq\theta^V(x)$.

In particular, combining this with the end of section \ref{section:C2Geo}, we have that at $||V||$-a.e. $x\in M_i$ the following four statements hold:
\begin{enumerate}
    \item $$\frac{n}{\alpha^{1/n}}F^{n/(n-1)}(x)\theta^V(x)\leq F(x)\theta^V(x)[\tr_{T_xM}(\hess^M(f)(x))+|H(M,x)|].$$
    \item $\hess^{M_i}(f)(x)=\hess^V(f)(x)$.
    \item $H(M_i,x)=H(V,x)$.
    \item $T_xM=T_xV$.
\end{enumerate}

Therefore at $||V||$-a.e. $x$ we have that $$\frac{n}{\alpha^{1/n}}F^{n/(n-1)}(x)\theta^V(x)\leq F(x)\theta^V(x)[\tr_{T_xV}(\hess^V(f)(x))+|H(V,x)|].$$

Using theorem \ref{thm:alex} we therefore deduce that \begin{align*}
   \frac{n}{\alpha^{1/n}} \int F^{n/(n-1)}(x)\intvar{x}&\leq\int F(x)[\tr_{T_xV}(\hess^V(f)(x))+|H(V,x)|]\intvar{x}
   \\
   &\leq\int -\grad^VF(x)\cdot\grad^V(f)(x)\intvar{x}+\int F(x)|H(V,x)|\intvar{x}
   \\
   &\leq\int |\grad^VF(x)|\intvar{x}+\int F(x)|H(V,x)|\intvar{x}.
\end{align*}
    By the same method as in \cite[Section~3]{Brendle_2023} we may choose a sequence of $\rho_i$ such that the corresponding $\frac{n}{\alpha_i^{1/n}}\rightarrow C_{n,m}$.

    The second equality in theorem \ref{Thm:Theorem1} follows since $$\int F(x)|H(V,x)|\intvar{x}+\int F(x)\dv||\delta V||_{sing}(x)=\int F(x)\dv||\delta V||(x)$$ and $$\int F(x)\dv||\delta V||_{sing}(x)=0.$$
\end{proof}

\section{Estimates at the Boundary}

We now use the theory of Allard\cite{Allard_1975bnd} to provide the estimates for theorem 3.

Namely, by \cite[Lemma~3.1(1)]{Allard_1975bnd} we have that $$||\partial V||(F)\coloneqq\lim_{r\downarrow0}\frac{1}{r}\int_{0<dist(x,\spt(||\delta V||_{sing}))<r} F\intvar{x}$$ exists for any non-negative and smooth $F$. In particular, for any such $F$ the functionals $||\partial V||^+(F)$ and $||\partial V||^-(F)$ are linear and equal.

This, combined with the fact that $\HH^n(\spt(||V||)\cap B)\leq\HH^n(B)=0$ is exactly the first part of theorem 3.

To see the second part, let $F$ is smooth and non-negative. We have 
\begin{lemma}
    If $G$ is a modified normal field for $(V,B)$, then $$\varepsilon_G||\partial V||(F)\leq\int_B F\dv||\delta V||=\int F\dv||\delta V||_{sing}.$$
\end{lemma}
\begin{proof}
    From the proof of \cite[Lemma~3.1(2)]{Allard_1975bnd}, and since $||V||(B)=0,\spt(||\delta V||_{sing})=B$ we find that by setting $g(x)=F(x)DG(x)$, the equality
    \begin{align*}
    \int FDG\cdot\dv(\delta V)&=\int_{\RR^{n+m}-B}FDG\cdot H(V)\dv||V||+\int FDG\cdot\dv(\delta V)_{sing}
        \\
        &=-\lim_{r\downarrow0}\frac{1}{r}\int_{0<dist(x,\spt(||\delta V||_{sing}))<r} F(x)DG(x)\cdot\grad^Vd_B(x)\intvar{x}
        \\
        &+\int_{\RR^{n+m}-B} FDG\cdot\dv(\delta V).
    \end{align*}

    Hence we have that $$-\lim_{r\downarrow0}\frac{1}{r}\int_{0<dist(x,\spt(||\delta V||_{sing}))<r} F(x)DG(x)\cdot\grad^Vd_B(x)\intvar{x}=\int FDG\cdot\dv(\delta V)_{sing}.$$

    By definition of $\varepsilon_G$, we find that $$-\varepsilon_G\lim_{r\downarrow0}\frac{1}{r}\int_{0<dist(x,\spt(||\delta V||_{sing}))<r} F(x)\intvar{x}\geq\int FDG\cdot\dv(\delta V)_{sing}.$$

    However, since the norm of $-FDG(x)$ at every point in $\spt(||\delta V||_{sing})$ is bounded by $F$, we conclude that $$\varepsilon_G||\partial V||(F)\leq-\int FDG\cdot\dv(\delta V)_{sing}\leq\int F\dv||\delta V||_{sing}.$$
\end{proof}

\section{Proof of Theorems \ref{Thm:Theorem2} and \ref{Thm:Theorem3}}

The results of the prior section imply that, provided we can establish theorem \ref{Thm:Theorem2}, theorem \ref{Thm:Theorem3} follows immediately.

\begin{proof}[Proof of Theorem \ref{Thm:Theorem2}]
    Suppose that $||\partial V||^+(1)<\infty$. Since any continuous function is bounded on any compact set, we have that $$||\partial V||^+(F)\leq[\sup_{x\in\spt(||V||)}|F(x)|]||\partial V||^+(1)<\infty.$$

    Since $||\partial V||^-(F)=-||\partial V||^+(-F)$ we also have that $$||\partial V||^-(1)\geq-[\sup_x|F(x)|]||\partial V||^+(1)>-\infty.$$

    Hence $||\partial V||^+(F),||\partial V||^-(F)$ are finite, and moreover $||\partial V||^+$ is a sublinear functional on the Banach space $C^0(\RR^{n+m})$.

    Let $\gamma^V$ be any functional dominated by $||\partial V||^+$, and Hahn-Banach implies there is at least, and notice that if we assume $F\geq0$ then $$\gamma^V(F)\geq-||\partial V||^+(-F)=||\partial V||^-(F)\geq0.$$

    Hence $\gamma^V$ is a positive linear functional on $C_c(\RR^{n+m})$, and thus it is a measure by Riesz. Moreover, the condition on the support is obvious.

    Now, for each $\varepsilon,r>0$ pick a smooth function $\beta_{r,\varepsilon}\geq0$ such that $\beta_{r,\varepsilon}=0$ on $[0,\varepsilon)$, $\beta_{r,\varepsilon}=1$ on $[r+5000\varepsilon,\infty)$ and $0\leq\beta_{r,\varepsilon}'\leq\frac{1}{(r+\varepsilon)}$.

    Let $F$ be non-negative and smooth, and set $$G(x)=F(x)[\beta_{r,\varepsilon}\circ dist(x,\spt(||\delta V||_{sing}))].$$ We find that since $G$ is then non-negative, smooth, and compactly supported away from $\spt(||\delta V||_{sing})$, by theorem \ref{Thm:Theorem1}, we have $$C_{n,m}\left(\int G^{n/(n-1)}\intvar{x}\right)^{(n-1)/n}\leq\int |\grad^VG|(x)+G(x)|H(V,x)|\intvar{x}.$$

    As $\varepsilon\rightarrow0$ this implies that \begin{align*}
        C_{n,m}\left(\int_{dist(x,\spt(||\delta V||))>r} F^{n/(n-1)}\intvar{x}\right)^{(n-1)/n}&\leq\int_{dist(x,\spt(||\delta V||))>0} |\grad^VF|+F|H(V,*)|\dv||V||
        \\
        &+\frac{1}{r}\int_{0<dist(x,\spt(||\delta V||))<r}F(x)\intvar{x}.
    \end{align*}

    Notice that $|\grad^VF|\leq |DF|$ and $H(V,x)\in L_{loc}^1(||V||)=L^1(||V||)$, since the varifold has locally bounded first variation and compact support, so that as $r\downarrow0$ the term on the left and the first term on the right converge to $$C_{n,m}\left(\int_{\RR^{n+m}-\spt(||\delta V||_{sing})} F^{n/(n-1)} \intvar{x}\right)^{(n-1)/n}$$ and $$\int_{\RR^{n+m}-\spt(||\delta V||_{sing})} |\grad^VF|(x)+F(x)|H(V,x)|\intvar{x}$$ respectively.

    On the other hand we know that $$\liminf_{r\downarrow0}\frac{1}{r}\int_{0<dist(x,\spt(||\delta V||))<r}F(x)\intvar{x}=||\partial V||^-(F)\leq\int F\dv\gamma^V.$$

    The conclusion follows.
\end{proof}

\begin{appendices}
\section{Convex Functions}\label{AppendixA}
    Let $F:\RR^n\rightarrow \RR$ be convex and let $G:B_1(0)\subseteq\RR^m\rightarrow\RR^n$ be $C^2$. Our goal is to show that $H=F\circ G$ is semiconvex with quadratic modulus.

    \begin{proposition}
        Let $F,G,H$ be as above, and $0<r<1$. Then there exists a constant $\lambda\in\RR$ such that $H(x)+\lambda|x|^2$ is convex on $B_{r}(0)$.
    \end{proposition}
    \begin{proof}
        Notice that $\overline{G(B_{r}(0))}$ is compact, hence $F|_{G(B_{r}(0))}$ is Lipschitz, of constant $L>0$.

        Upon mollifying $F$ we define $H_\varepsilon=F_\varepsilon\circ G$ and notice that $H_\varepsilon\rightarrow H$ uniformly on $B_{r}(0)$. Furthermore, $\lip(F_\varepsilon|_{G(B_{r}(0))})\leq L$ and $F_\varepsilon$ is convex.

        Now, we have that $\partial_iH_\varepsilon=\sum_k\partial_kF_\varepsilon(G(x))\partial_iG^k$ and $\partial^2_{ij}H_\varepsilon=\sum_{k,l}\partial^2_{kl}F_\varepsilon(G(x))\partial_iG^k\partial_jG^l+\sum_{k}\partial_{k}F_\varepsilon(G(x))\partial^2_{ij}G$

        The second term is bounded by the $C^2$ norm of $G$, the dimensions $m$ and $n$, and $L$ times the identity matrix. Meanwhile for $v\in\RR^m$ we have that $$v^T(\partial^2_{ij}H_\varepsilon+C)v\geq v^T(\partial^2_{ij}F_\varepsilon)v\geq0$$
    \end{proof}

    By the theory of convex functions, namely \cite[Theorem~14.25]{Villani_2009}, it holds that for $\LL^n$-a.e. $x\in B_{r}(0)$ there is a Taylor expansion $$H(x+v)=H(x)+\grad H(x)\cdot v+\langle Av,v\rangle+o(|v|^2).$$

    Consider now a $C^2$ diffeomorphism $$\Phi(x+v)=\Phi(x)+D\Phi(x)\cdot v+D^2\Phi(x)(v,v)+o(|v|^2)$$ near $x$. It is clear that $$H(\Phi(x+v))=H(\Phi(x)+D\Phi(x)\cdot v+D^2\Phi(x)(v,v))+o(|v|^2)=H(\Phi(x)+[D\Phi(x)\cdot v+D^2\Phi(x)(v,v)])+o(|v|^2).$$ Hence, if we set $\tilde H(x+v)=H(\Phi(x+v))$ there is a Taylor expansion $$\tilde H(x+v)=\tilde H(x)+\grad \tilde H(x)\cdot v+\langle \tilde Av,v\rangle+o(|v|^2).$$

    If also $\hat G$ is a $C^2$ function such that $|\hat{G}(x+v)-G(x+v)|=o(|v|^2)$ then it is clear that $$F\circ\hat{G}(x+v)=F\circ G(x+v)+o(|v|^2)=H(x)+\grad H(x)\cdot v+\langle Av,v\rangle+o(|v|^2).$$

    In summary, we have \begin{proposition}
        Let $M$ be a $C^2$ submanifold on $\RR^n$, and let $G:B_1(0)\subseteq\RR^m\rightarrow\RR^n$ be a $C^2$ coordinate parameterization of $M$ near some point $p$. Then for $\LL^m$-a.e. $x\in B_1(0)$, $F\circ G$ has a second order Taylor expansion at $x$. 
        
        Moreover, if $H$ is another $C^2$ parameterization with $G(x)=H(y),y\in B_1(0)$, then $F\circ H$ has a second order Taylor expansion at $y$.

        If, instead, $H$ is a $C^2$ function such that the 2-jet of $G$ and $H$ at $x$ are equal, then $F\circ H$ has a second order Taylor expansion at $x$, and moreover, the Taylor expansion of $F\circ H$ is equal to that of $F\circ G$ at $x$.
    \end{proposition}

\section{Graph Coordinates}\label{AppendixB}

We recall some properties of graph coordinates which are, to the author's knowledge, folklore results:
\begin{theorem}
    Let $M^n\subseteq\RR^{n+m}$ be a $C^2$ submanifold and let $p\in M$. Then there exists an open ball $B^n_r(0)\subseteq\RR^n$, an open set $U\ni p$, an orthogonal transformation $O$ and a $C^2$ function $u:B^n_r(0)\rightarrow\RR^m$ such that:
    \begin{enumerate}
        \item $M^n\cap U=F(B^n_r(0))$, where $F(x)=O\circ(x,u(x))+p$.
        \item $F(0)=p$.
        \item $Du(0)=0$ and $T_pM^n=O(\RR^n\times\{0\}^m)=DF(0)\RR^n$.
        \item $\sff_p(DF(0)v,DF(0)v)=\hess(F(0))(v,v)$.
        \item In local coordinates the metric $g$ has the Taylor expansion $g(v)=Id+o(|v|)$.
    \end{enumerate}
\end{theorem}
These facts are relatively standard and follow immediately from the implicit function theorem.

Combining this with appendix A we get:
\begin{corollary}
     Let $M^n\subseteq\RR^{n+m}$ be a $C^2$ submanifold, let $G:\RR^{n+m}\rightarrow\RR$ be locally convex and let $p\in M$. Then for $\HH^n$-a.e. $x\in M^n$ there exists a symmetric bilinear form $A_x\in T^*_xM^n\odot T^*_xM^n$ such that $$G(x+v+\sff_x(v,v))=G(x)+\grad^{M^n}G(x)\cdot v+A_x(v,v)+o(|v|^2),v\in T_xM^n$$
\end{corollary}
\begin{proof}
    We first cover $M^n$ with a countable family of $C^2$ coordinate charts $\phi_i:B_1(0)\rightarrow \RR^{n+m}$ such that $\phi_i(B_{1/2}(0))$ still covers the manifold.

    Then by appendix \ref{AppendixA} for $\HH^n$-a.e. $x\in B_{1/2}(0)$ there is a second order Taylor expansion for $G\circ\phi_i$ at $x$. Moreover, we see that if $\psi$ is another coordinate chart such that $\psi(y)=\phi_i(x)$ then $G\circ\psi$ has a second order Taylor expansion at $y$.

    By taking $\psi=F$ as graph coordinates at $\psi_i(x)$, we deduce that there is a second order expansion of $G\circ F$ at $0$. But since, setting $w=DF(0)^{-1}v$ $F(w)=p+DF(0)w+\sff_p(DF(0)w,DF(0)w)+o(|w|^2)$ we get that $1/C|v|\leq|w|\leq C|v|$ and therefore \begin{align*}
        G(p+DF(0)w+\sff_p(DF(0)w,DF(0)w))&=G(F(w))+o(|w|^2)
        \\
        &=G(F(0))+\grad (G(F(w)))(0)\cdot w+\langle Aw,w\rangle+o(|w|^2)
        \\
        &=G(p)+\grad^{M^n}G(p)\cdot DF(0)DF(0)^{-1}v
        \\
        &+\langle A[DF(0)^{-1}]v,DF(0)^{-1}v\rangle+o(|v|^2).
    \end{align*}

    Hence if we set $A_x(v,v)=\langle A[DF(0)^{-1}]v,DF(0)^{-1}v\rangle$ the conclusion follows at $\phi_i(x)$. In particular, since each $\phi$ is a $C^2$ coordinate chart, and the number of charts is countable, this holds for $\HH^n$-a.e. $x\in M^n$.
\end{proof}

We will write $\hess^{M^n}(G)(x)=A_x$ when it exists, and extend to $\RR^{n+m}$ by taking $$\hess^{M^n}(G)(x)(w,w)= \hess^{M^n}(G)(x)(\tan(M^n,x)w,\tan(M^n,x)w).$$

We also identify it with the linear map $\RR^{n+m}\rightarrow T_xM^n\subseteq\RR^{n+m}$ in the natural way.

We now let $N$ be the smooth manifold defined by $v\mapsto x+v+\sff_x(v,v)$. Notice here that if $\hess^{M^n}(G)(x)$ exists, then $\hess^{N}(G)(x)$ exists, they are equal, and moreover $\grad^NG(x)=\grad^{M^n}G(x)$.

\begin{lemma}
    $\hess^{M^n}(G)(x)$ has the property that $$G(y)=G(x)+\grad^{M^n}G(x)\cdot (y-x)+\hess^{M^n}(G)(x)((y-x),(y-x))+o(d_{N}^2(x,y)),y\in N$$ and $$G(y)=G(x)+\grad^{M^n}G(x)\cdot (y-x)+\hess^{M^n}(G)(x)((y-x),(y-x))+o(d_{M^n}^2(x,y)),y\in M^n.$$
\end{lemma}
\begin{proof}
    Notice that if we write $y=x+v+\sff_x(v,v)$ then $\grad^{M^n}G(x)\cdot\sff_x(v,v)=0$, so $\grad^{M^n}G(x)\cdot (y-x)=\grad^{M^n}G(x)\cdot v$, and similarly $\hess^{M^n}(G)(x)((y-x),(y-x))=\hess^{M^n}(G)(x)(v,v)$.

    If we recall that nearby $x$ the intrinsic and extrinsic distance are equivalent, this reduces to the statement that $$G(y)=G(x)+\grad^{M^n}G(x)\cdot v+\hess^{M^n}(G)(x)(v,v)+o(|v+\sff_x(v,v)|^2).$$

    On the other hand, $v+\sff_x(v,v)=v+O(|v|^2)$, so $o(|v+\sff_x(v,v)|^2)=o(|v|^2)$, and thus this reduces to the statement that $$G(y)=G(x)+\grad^{M^n}G(x)\cdot v+\hess^{M^n}(G)(x)(v,v)+o(|v|^2)$$ which is true by definition of $\hess^{M^n}(G)(x)(v,v)$.

    To conclude the second statement, consider graph coordinates $F$ at $x$. Then since $F$ is bilipschitz onto its image near $x$, and since $F(w)= p+DF(0)w+\sff_p(DF(0)w,DF(0)w)+o(|w|^2)$, the first Taylor expansion implies the second.
\end{proof}

Putting everything together, we deduce

\begin{proposition}
\label{prop:smoothapprox}
    Suppose that $\hess^{M^n}(G)(x)$ exists. Then there exists a smooth function $u$ such that $|u(y)-G(y)|=o(d^2_N(x,y)),y\in N$ and $|u(y)-G(y)|=o(d^2_{M^n}(x,y)),y\in M^n$. Moreover, we may take $$u(y)=G(x)+\grad^{N}G(x)\cdot (y-x)+\hess^{N}(G)(x)((y-x),(y-x)).$$
\end{proposition}

\end{appendices}
\printbibliography

\end{document}